\documentclass[12pt]{amsart}
\usepackage{amssymb, amsmath, amsthm, latexsym, array,euscript}
\usepackage{fullpage}
\usepackage{verbatim}

\usepackage{xcolor}
\definecolor{darkblue}{rgb}{0,0,.5}
\definecolor{darkgreen}{rgb}{.2,0.5,.2}
\usepackage[colorlinks=true,linkcolor=darkblue,urlcolor=violet,citecolor=magenta]{hyperref}

\numberwithin{equation}{section}

\input{cyracc.def}

\font\tencyr=wncyr10 %scaled \magstephalf
\font\tencyi=wncyi10 %scaled \magstephalf
\font\tencysc=wncysc10 %scaled \magstephalf

\def\rus{\tencyr\cyracc}
\def\rusi{\tencyi\cyracc}
\def\rusc{\tencysc\cyracc}

\newtheorem{thm}{Theorem}[section]
\newtheorem{lm}[thm]{Lemma}%[section]
\newtheorem{cl}[thm]{Corollary}
\newtheorem{prop}[thm]{Proposition}

\theoremstyle{remark}
\newtheorem{ex}[thm]{Example}%[section]
\newtheorem{rmk}[thm]{Remark}%[section]
\theoremstyle{definition}
\newtheorem{df}[thm]{Definition}%%[section]

\newcommand {\g}{{\mathfrak g}}
\newcommand {\h}{{\mathfrak h}}

\newcommand {\q}{{\mathfrak q}}
\newcommand {\rr}{{\mathfrak r}}

\newcommand {\te}{{\mathfrak t}}

\newcommand {\z}{{\mathfrak z}}

\newcommand {\sln}{{\mathfrak {sl}}_n}

\newcommand {\eus}{\EuScript}

\newcommand {\gA}{{\eus A}}

\newcommand {\gN}{{\eus N}}

\newcommand {\gS}{{\eus S}}
\newcommand {\gZ}{{\eus Z}}

\newcommand {\vp}{\varphi}
\newcommand {\vth}{\vartheta}

\newcommand{\gt}{\mathfrak}

\newcommand{\GL}{{\rm GL}}

\newcommand{\Aut}{\mathsf{Aut}}
\newcommand{\Ann}{\mathrm{Ann}}
\newcommand{\ad}{\mathrm{ad}}

\newcommand{\id}{{\rm id}}
\newcommand{\ind}{{\rm ind\,}}

\newcommand{\codim}{\mathrm{codim\,}}
\newcommand{\rk}{\mathrm{rk\,}}

\newcommand{\Lie}{\mathsf{Lie\,}}

\newcommand{\spe}{\mathsf{Spec\,}}

\newcommand{\gr}{\mathrm{gr\,}}
\newcommand{\oth}{\mathsf{ord}(\vartheta)}

\newcommand {\ca}{{\mathcal A}}

\newcommand {\cN}{{\mathcal N}}

\newcommand {\cz}{{\mathcal Z}}

\newcommand{\wa}{\tilde{{\mathcal A}}}

\newcommand {\trdeg}{{\mathrm{tr.deg\,}}}

\newcommand {\mK}{{\Bbbk}}
\newcommand {\bbk}{{\Bbbk}}

\newcommand {\Z}{{\mathbb Z}}

\newcommand {\md}{/\!\!/}
\newcommand {\isom}{\stackrel{\sim}{\rightarrow}}
\newcommand {\GR}[2]{{\mathsf{{#1}}}_{#2}}

\newcommand {\BZ}{{\mathbb Z}}
\newcommand {\BN}{{\mathbb N}}
\newcommand {\beq}{\begin{equation}}
\newcommand {\eeq}{\end{equation}}

\newcommand {\hm}{{\h\langle m\rangle}}
\newcommand {\hmi}{{\h\langle m{+}1\rangle}}
\newcommand {\qqn}{{\q\langle n\rangle}}

\newcommand {\rg}{{\rangle}}
\renewcommand {\lg}{{\langle}}

\renewcommand{\le}{\leqslant}
\renewcommand{\ge}{\geqslant}

\newcommand {\BP}{{\mathbb P}}
\newcommand{\U}{{\eus U}}

\newcommand{\tr}{ {\rm tr}}

\newcommand{\bb}{\boldsymbol{b}}
\newcommand{\bF}{\boldsymbol{F}\!}
\newcommand{\bff}{\boldsymbol{f}}

\newcommand {\cG}{{\mathcal G}}
\newcommand {\Gz}{{\mathcal G}(\vec z)}

\newcommand{\wg}{{\widehat{\h }}}
\newcommand{\wh}{{\widehat{\h}}}

\begin{document}
\hfill {\scriptsize February 19, 2021}
\vskip1ex

\title[Periodic automorphisms and compatible brackets]{Periodic automorphisms, compatible Poisson brackets, and Gaudin subalgebras}
\author[D.\,Panyushev]{Dmitri I. Panyushev}
\address[D.P.]%
{Institute for Information Transmission Problems of the R.A.S, Bolshoi Karetnyi per. 19,
Moscow 127051, Russia}
\email{panyushev@iitp.ru}
\author[O.\,Yakimova]{Oksana S.~Yakimova}
\address[O.Y.]{Institut f\"ur Mathematik, Friedrich-Schiller-Universit\"at Jena,  07737 Jena,
Deutschland}
\email{oksana.yakimova@uni-jena.de}
\thanks{The first author is funded by RFBR, project {\rus N0} 20-01-00515.
The second author is funded by the DFG (German Research Foundation) --- project number 404144169.}
\dedicatory{To the memory of our teacher Ernest Borisovich Vinberg}
\maketitle

\tableofcontents
%%%%%%%%%%%   Introduction  %%%%%%
\section*{Introduction}
\label{sect:intro}

\noindent
The ground field $\mK$ is algebraically closed and $\mathsf{char}(\mK)=0$. 
%In this article, we prove quite a few marvellous outstanding results that are of great great importance for the future development of Mathematics.
Let $\q=(\q, [\,\,,\,])$ be a finite-dimensional algebraic Lie algebra, i.e., $\q=\Lie Q$, where $Q$ is a 
connected affine algebraic group. The dual space $\q^*$ is a Poisson variety, i.e., the 
algebra of polynomial functions on $\q^*$, $\bbk[\q^*]\simeq \gS(\q)$, is equipped with the Lie--Poisson 
bracket $\{\,\,,\,\}$. Here $\{x,y\}=[x,y]$ for $x,y\in\q$. 
Poisson-commutative subalgebras of $\bbk[\q^*]$ are  important tools for the study of geometry of
the coadjoint representation of $Q$ and representation theory of $\q$. 

\subsection{}    \label{subs:old}
There is a well-known method, {\it the Lenard--Magri scheme}, for constructing ``large" 
Poisson-commutative subalgebras of $\bbk[\q^*]$, which is related to {\it compatible Poisson brackets}, 
see e.g.~\cite{GZ}.  
Two Poisson brackets $\{\,\,,\,\}'$ and $\{\,\,,\,\}''$ are said to be {\it compatible}, if any linear combination 
$\{\,\,,\,\}_{a,b}:=a\{\,\,,\,\}'+b\{\,\,,\,\}''$ with $a,b\in\bbk$ is a Poisson bracket. Then one defines a certain 
dense open subset $\Omega_{\sf reg}\subset\bbk^2$ that corresponds to the {\it regular} brackets in the 
pencil $\mathcal P=\{ \{\,\,,\,\}_{a,b}\mid (a,b)\in\bbk^2\}$.  
If $\cz_{a,b}\subset \gS(\q)$ denotes the Poisson centre of  $(\gS(\q), \{\,\,,\,\}_{a,b})$, then
the subalgebra $\gZ\subset\gS(\q)$ generated by  $\cz_{a,b}$
with $(a,b)\in\Omega_{\sf reg}$ is Poisson-commutative w.r.t.~$\{\,\,,\,\}'$ and $\{\,\,,\,\}''$, see 
Section~\ref{subs:construct-2} for more details. An obvious first step is to take the initial Lie--Poisson 
bracket $\{\,\,,\,\}$ as $\{\,\,,\,\}'$. The rest depends on a clever choice of $\{\,\,,\,\}''$.
The goal of this article is to introduce a new class of compatible
Poisson brackets, study the respective subalgebras $\gZ$, and provide applications.

Let us recall some known pencils of compatible Poisson brackets. 
\\ \indent
\textbullet \quad For any $\gamma\in\q^*$, the Poisson bracket    
$(x,y)\mapsto \{x,y\}_{\gamma}:=\gamma([x,y])$, $x,y\in\q$, is compatible with~$\{\ ,\,\}$, 
cf.~\cite[Example~1.8.16]{duzu}. This leads to the fabulous {\it Mishchenko--Fomenko (= MF) 
subalgebras} $\gZ={\ca}_\gamma\subset\gS(\q)$, which first appear in~\cite{mf} for semisimple 
$\q$. To the best of our knowledge,  the very term ``Mishchenko--Fomenko'' is coined by 
Vinberg~\cite{vi90}.
\\ \indent
\textbullet \quad In~\cite{OY}, we introduced compatible Poisson brackets related to a $\BZ_2$-grading 
$\q=\q_0\oplus\q_1$ and studied the respective subalgebra $\gZ=\gZ(\q,\q_0)$.
Here the second bracket is defined by the relations $\{x,y\}''=\{x,y\}$ if $x\in\q_0$ and $\{x,y\}''=0$ if $x,y\in\q_1$. As well as with 
MF subalgebras, one cannot get too far if $\q$ is arbitrary.
Assuming that $\q=\g$ is reductive, we obtained a number of interesting results on $\gZ(\g,\g_0)$.
We proved that:
\\ \indent
-- $\gZ(\g,\g_0)$ is a  Poisson-commutative subalgebra of $\gS(\g)^{\g_0}$ having the maximal 
possible transcendence degree, which equals $\trdeg \gZ(\g,\g_0)=(\dim\g_1+\rk\g+\rk \g_0)/2$; 
\\ \indent
-- with only four exceptions related to exceptional Lie algebras,  $\gZ(\g,\g_0)$ is a polynomial algebra whose algebraically independent generators are explicitly described;
\\ \indent
--  if  $\g$ is a classical Lie algebra and $\g_0$ contains a regular nilpotent element of $\g$, then  
$\gZ(\g,\g_0)$ is a maximal Poisson-commutative subalgebra of $\gS(\g)^{\g_0}$.
\\
The proofs exploit numerous invariant-theoretic properties of the adjoint representation of 
$\g$~\cite{ko63} and their analogues for the isotropy representation $G_0\to {\rm GL}(\g_1)$~\cite{kr71}.

\subsection{}    \label{subs:new}
Results of the present paper stem from a surprising observation that if $\q$ is equipped with
a $\BZ_m$-grading with any $m\ge 2$, then one can naturally construct a compatible Poisson bracket
$\{\,\,,\,\}''$ (Section~\ref{sect:vartheta-&-compatible}). In this case, all Poisson brackets in $\mathcal P$
are linear and there are two lines $l_1,l_2\subset \bbk^2$ such that
$\Omega=\bbk^2\setminus (l_1\cup l_2)\subset \Omega_{\sf reg}$ and the Lie algebras corresponding to
$(a,b)\in \Omega$ are isomorphic to $\q$. The lines $l_1$ and $l_2$ give rise to new Lie algebras, denoted $\q_{(0)}$ and $\q_{(\infty)}$. These new algebras are different contractions of $\q$. Let 
$\ind\q$ denote the {\it index\/} of $\q$ (see Section~\ref{sect:prelim}). Then $\ind\q\le\ind\q_{(n)}$,
$n\in\{0,\infty\}$, and our first task is to realise whether it is true that $\ind\q=\ind\q_{(n)}$.
Although basic theory can be developed for arbitrary $\q$, essential applications require 
a better class of Lie algebras, and we eventually stick to the semisimple case. Let 
$\vth\in\mathsf{Aut}(\g)$ be an automorphism of order $m$ and $\g=\bigoplus_{j=0}^{m-1}\g_j$ the 
associated  $\BZ_m$-grading, i.e., if $\zeta=\sqrt[m]1$ is primitive, then 
\[
 \g_j=\{x\in\g\mid \vth(x)=\zeta^j x\} .
\] 
The invariant-theoretic base for our consideration is 
Vinberg's theory of ``$\vth$-groups'' (i.e., theory of orbits and invariants for representations of reductive
groups related to the periodic automorphisms of $\g$), see~\cite{vi76,vi79}.

A bad news is that, for $m\ge 3$, the invariant-theoretic picture related to $\vth$ and properties of
$\g_{(0)}$ become more complicated. For instance, if $m=2$, then 
$\g_{(0)}\simeq \g_0\ltimes\g_1^{\sf ab}$ (semi-direct product) and it is known that here 
$\ind\g_{(0)}=\ind\g=\rk\g$~\cite{p07}. For $m\ge 3$, the number $\ind\g_{(0)}$ remains mysterious. 
We suspect that it is equal to $\rk\gt g$ for any $\vth$. Other technical difficulties are discussed in 
Section~\ref{subs:g-null}. Nevertheless, we succeeded in computing $\ind\g_{(\infty)}$, see 
Theorem~\ref{thm:ind-inf}, and can state that $\ind\g_{(\infty)}=\rk\g$ if and only if $\g_0$ is abelian. 

A good news is that there are still many interesting cases (automorphisms $\vth$), where analogues 
of results of~\cite{OY} are valid and also some unexpected applications pop up.  Write $\gZ(\g,\vth)$ for 
the Poisson-commutative subalgebra associated with the pencil of compatible Poisson brackets related 
to $\vth$. Let $\cN$ be the set of nilpotent elements of $\g$. Suppose that $\vth$ has the following 
properties:
\\ \indent
(1) \ $\g_1$ contains a regular semisimple element of $\g$; 
\\ \indent
(2) \ each irreducible component of
$\g_1\cap\cN$ contains a regular nilpotent element of $\g$.
\\
Under these assumptions we prove that $\gZ(\g,\vth)$ has the same properties as the above
algebra $\gZ(\g,\g_0)$, see Sections~\ref{sect:free}, \ref{sect:max}.

\subsection{} 
Another good news is that there is a special case of $(\g,\vth)$, when $\gZ(\g,\vth)$ is as 
good as possible and it has a nice quantisation. Namely, let $\h$ be a non-abelian simple Lie algebra 
and $\g=\h^{m}$ be the direct sum of $m$ copies of $\h$.  If $\vth$ is the cyclic permutation of the 
summands, then $\g_0=\Delta_\h\simeq\h$ and we prove in Section~\ref{sect:cyc} that  
$\gZ=\gZ(\h^{m},\vth)$ is a polynomial ring in  $\frac{1}{2}((m-1)\dim\h+(m+1)\rk\h)$ generators. 
Furthermore, $\gZ$ is a maximal Poisson-commutative subalgebra of $\gS(\g)^{\g_0}$. 
The quantisation problem asks for a lift of $\gZ$ to the enveloping algebra $\U(\g)$, i.e., for a commutative 
algebra $\tilde \gZ\subset \U(\g)$ such that $\gr(\tilde\gZ)=\gZ$. To describe $\tilde\gZ$ in this context, 
we need some preparations.

The enveloping algebra $\U(\h [t,t^{-1}])$ of the loop algebra $\h[t,t^{-1}]$ contains a large 
commutative subalgebra $\z(\widehat{\h })$ of infinite transcendence degree,  known as the 
{\it Feigin--Frenkel centre}~\cite{ff:ak}. Actually $\z(\wh)\subset \U(\wg_-)$, where  
$\wg_-=t^{-1}\h [t^{-1}]$. Having a vector $\vec z\in(\bbk^\star)^m$, one  defines  
a  homomorphism $\rho_{\vec z}\!:\U(\wg_-) \to \U(\g)$. The image of $\z(\wh)$ under $\rho_{\vec z}$  
is called the  {\it Gaudin subalgebra} $\Gz$ \cite{FFRe}. 
For the case, where the entries of $\vec z$ are pairwise distinct $m$-th roots of unity,
we provide a simpler construction of $\Gz$ (Proposition~\ref{prop:new-G})
and show that  $\gr\!(\Gz)=\gZ(\h^{m},\vth)$, see Theorem~\ref{q}.

It is worth noting that, for the MF subalgebras $\ca_\gamma\subset \gS(\h)$, the quantisation problem 
was posed by Vinberg~\cite{vi90}. 
A solution given by Rybnikov~\cite{r:si} states that the image of $\z(\wh)$ under a certain homomorphism 
$\varrho_\gamma : \U(\wh_-)\to \U(\h)$, depending on $\gamma\in\h^*$, is the 
{\it quantum  MF subalgebra} $\wa_\gamma$ and one has $\gr(\wa_\gamma)=\ca_\gamma$ in many 
cases.

Let us identify $\wg_-$ with the quotient space $\h [t,t^{-1}]/\h [t]$.  Then $\h [t]$ acts on 
$t^{-1}\h [t^{-1}]$ and correspondingly on $\gS(\wg_-)$.  By~\cite{ff:ak,f:lc}, one has  
\beq           \label{hat-inv}
         \gr\!(\z(\wh))=\gS(\wg_-)^{\h[t]} .
\eeq 
Yet another property of the Poisson-commutative algebra $\gr\!(\z(\wh))$ is that it is a polynomial ring 
in infinitely many variables, by a direct generalisations of a Ra\"is--Tauvel theorem~\cite{rt}. 

Suppose that $\vth\in\Aut(\h)$. In Section~\ref{sec-twist}, we consider the $\vth$-twisted version of 
$\gr\!(\z(\wh))$, a certain subalgebra $\gZ(\wh_-,\vth)$ of $\gS((\wg_-)^\vth)$ of infinite transcendence 
degree.  Assuming the equality $\ind\h_{(0)}=\rk\h$, we prove that $\gZ(\wh_-,\vth)$ is 
Poisson-commutative, see Theorem~\ref{twist-t}. In many cases, $\gZ(\wh_-,\vth)$ is a polynomial ring. 

Our general reference for semisimple Lie groups and algebras is \cite{t41}.

%%%%%%%%%%  Section  
\section{Preliminaries on Poisson brackets and polynomial contractions}
\label{sect:prelim}

\noindent
Let $Q$\/ be a connected affine algebraic group with Lie algebra $\q$. The  symmetric algebra of 
$\q$ over $\mK$ is $\BN$-graded, i.e., $\gS(\q)=\bigoplus_{i\ge 0}\gS^i(\q)$. It is identified with the 
algebra of polynomial functions on the dual 
space $\q^*$, and we also write $\mK[\q^*]=\bigoplus_{i\ge 0}\bbk[\q^*]_i$ for it. 
\subsection{The coadjoint representation}
\label{subs:coadj}
The group $Q$ acts on $\q^*$ via the coadjoint representation and then $\ad^*\!: \q\to \mathrm{GL}(\q^*)$ 
is the {\it coadjoint representation\/} of $\q$. The algebra 
of $Q$-invariant polynomial functions on $\q^*$ is denoted by $\gS(\q)^{Q}$ or $\mK[\q^*]^Q$.
Write $\mK(\q^*)^Q$ for the field of $Q$-invariant rational functions on $\q^*$.
\\ \indent
Let $\q^\xi=\{x\in\q\mid \ad^*(x){\cdot}\xi=0\}$ be the {\it stabiliser\/} in $\q$ of $\xi\in\q^*$. The 
{\it index of\/} $\q$, $\ind\q$, is the minimal codimension of $Q$-orbits in $\q^*$. Equivalently,
$\ind\q=\min_{\xi\in\q^*} \dim \q^\xi$. By the Rosenlicht theorem (see~\cite[IV.2]{spr}), one also has 
$\ind\q=\trdeg\mK(\q^*)^Q$. Set $\bb(\q)=(\dim\q+\ind\q)/2$. 
Since the $Q$-orbits in $\q^*$ are even-dimensional, $\bb(\q)$ is an integer. If $\q$ is reductive, then  
$\ind\q=\rk\q$ and $\bb(\q)$ equals the dimension of a Borel subalgebra.  

The Lie--Poisson bracket in $\gS(\q)$ is defined on $\gS^1(\q)=\q$ by $\{x,y\}:=[x,y]$. It is then extended 
to higher degrees via the Leibniz rule. Hence $\gS(\q)$ has the usual associative-commutative structure 
and additional Poisson structure. Whenever we refer to {\sl subalgebras\/} of $\gS(\q)$, we always mean 
the associative-commutative structure. Then a subalgebra $\ca\subset \gS(\q)$ is said to be 
{\it Poisson-commutative}, if $\{H,F\}=0$ for all $H,F\in\ca$. It is well known that if $\ca$ is 
Poisson-commutative, then $\trdeg\ca\le \bb(\q)$, see e.g.~\cite[0.2]{vi90}. More generally, suppose that 
$\h\subset\q$ is a Lie subalgebra and $\ca\subset  \gS(\q)^\h$ is Poisson-commutative. Then
$\trdeg\ca\le \bb(\q)-\bb(\h)+\ind\h$, see~\cite[Prop.\,1.1]{m-y}.

The {\it centre\/} of the Poisson algebra $(\gS(\q), \{\,\,,\,\})$ is 
\[
    \cz(\q):=\{H\in \gS(\q)\mid \{H,F\}=0 \ \ \forall F\in\gS(\q)\} .
\]
Using the Leibniz rule, we obtain that $\cz(\q)$ is a graded Poisson-commutative subalgebra of $\gS(\q)$, which coincides with the algebra of symmetric invariants of $\q$, i.e.,
\[
    \cz(\q)=\{H\in \gS(\q)\mid \{H,x\}=0 \ \ \forall x\in\q\}=:\gS(\q)^\q=\bbk[\q^*]^\q .
\] 
As $Q$ is connected, we have $\gS(\q)^\q=\gS(\q)^{Q}=\mK[\q^*]^Q$. Since the quotient field of
$\mK[\q^*]^Q$ is contained in $\bbk(\q^*)^Q$, we deduce from the Rosenlicht theorem that
\beq    \label{eq:neravenstvo-ind}
    \trdeg (\gS(\q)^\q)\le \ind\q .
\eeq
The set of $Q$-{\it regular\/} elements of $\q^*$ is 
\beq       \label{eq:regul-set}
    \q^*_{\sf reg}=\{\eta\in\q^*\mid \dim \q^\eta=\ind\q\}=\{\eta\in\q^*\mid \dim Q{\cdot}\eta \ \text{ is maximal}\} .
\eeq
It is a dense open subset of $\q^*$. Set $\q^*_{\sf sing}=\q^*\setminus \q^*_{\sf reg}$.
We say that $\q$ has the {\sl codim}--$n$ property if $\codim \q^*_{\sf sing}\ge n$. 
The {\sl codim}--$2$ property  is going to be most important for us. 

For $\gamma\in\q^*$, let $\hat\gamma$ be the skew-symmetric bilinear form on $\q$ defined by 
$\hat\gamma(\xi,\eta)=\gamma([\xi,\eta])$ for $\xi,\eta\in\q$. It follows that
$\ker\hat\gamma=\q^\gamma$. The $2$-form $\hat\gamma$ is related to 
the {\it Poisson tensor (bivector)} $\pi$ of the Lie--Poisson bracket $\{\,\,,\,\}$ as follows.

Let $\textsl{d}H$ denote the differential of $H\in \gS(\q)=\bbk[\q^*]$. Then 
$\pi$ is defined by the formula
$\pi(\textsl{d}H\wedge \textsl{d}F)=\{H,F\}$ for $H,F\in\gS(\q)$. Then 
$\pi(\gamma)(\textsl{d}_\gamma H\wedge \textsl{d}_\gamma F)=\{H,F\}(\gamma)$ and therefore
$\hat\gamma=\pi(\gamma)$.
In this terms, $\ind\q=\dim\q-\rk\pi$, where $\rk\pi=\max_{\gamma\in\q^*}\rk\pi(\gamma)$. 

\subsection{Contractions and invariants}
\label{subs:contr-&-inv} 
We refer to \cite[Ch.\,7,\,\S\,2]{t41} for basic facts on contractions of Lie algebras.
In this article, we consider contractions of the following form. Let $\bbk^\star=\bbk\setminus\{0\}$ be 
the multiplicative group of $\bbk$ and 
$\vp: \bbk^\star\to \GL(\q)$, $s\mapsto \vp_s$, a polynomial representation. That is, 
the matrix entries of $\vp_{s}:\q\to \q$ are polynomials in $s$ w.r.t. some (any) basis of $\q$.
Define a new Lie algebra structure on the vector space $\q$ and associated Lie--Poisson bracket by 
\beq       \label{eq:fi_s}
      [x, y]_{(s)}=\{x,y\}_{(s)}:=\vp_s^{-1}[\vp_s( x), \vp_s( y)], \ x,y \in \q, \ s\in\bbk^\star.
\eeq
The corresponding Lie algebra is denoted by $\q_{(s)}$. Then $\q_{(1)}=\q$ and all these algebras 
are isomorphic. The induced $\bbk^\star$-action in the variety of structure constants in not necessarily 
polynomial, i.e., \ $\lim_{s\to 0}[x, y]_{(s)}$ may not exist for all $x,y\in\q$. Whenever such a limit exists, 
we obtain a new linear Poisson bracket, denoted $\{\,\,,\,\}_0$, and thereby a new Lie algebra $\q_{(0)}$, 
which is said to be a {\it contraction\/} of $\q$. If we wish to stress that this construction is determined 
by $\vp$, then we write $\{x, y\}_{(\vp,s)}$ for the bracket in~\eqref{eq:fi_s} and say that $\q_{(0)}$ is the 
$\vp$-{\it contraction\/} of $\q$ or is the {\it zero limit of $\q$ w.r.t.}~$\vp$.  
A criterion for the existence of $\q_{(0)}$
can be given in terms of Lie brackets of the $\vp$-eigenspaces in $\q$, see~\cite[Section\,4]{Y-imrn}. 
We identify all algebras $\q_{(s)}$ and 
$\q_{(0)}$ as vector spaces. The semi-continuity of index implies that $\ind\q_{(0)}\ge \ind\q$.

The map $\vp_s$, $s\in\bbk^\star$, is naturally extended to an invertible transformation of 
$\gS^j(\q)$, which we also denote by $\vp_s$. The resulting graded map 
$\vp_s:\gS(\q)\to\gS(\q)$ is nothing but the comorphism associated with $s\in\bbk^\star$ and the dual 
representation $\vp^*:\bbk^\star\to \GL(\q^*)$. Since $\gS^j(\q)$ has a basis that consists of 
$\vp(\bbk^\star)$-eigenvectors, any $F\in\gS^j(\q)$ can be written as $F=\sum_{i\ge 0}F_i$, 
where the sum is finite and $\vp_s(F_i)=s^iF_i\in\gS^j(\q)$. Let $F^\bullet$ denote the nonzero 
component $F_i$ with maximal $i$.

\begin{prop}[{\cite[Lemma~3.3]{contr}}]     \label{prop:bullet}
If $F\in\cz(\q)$ and $\q_{(0)}$ exists, then $F^\bullet\in \cz(\q_{(0)})$. 
\end{prop}

%%%%%%%%%%%   Section 2  %%%%%%
\section{Automorphisms of finite order and compatible Poisson brackets}
\label{sect:vartheta-&-compatible}

\noindent
In this section, we associate a pencil of compatible Poisson brackets to any automorphism of finite order
of a Lie algebra $\q$, describe the limit algebras $\q_{(0)}$ and $\q_{(\infty)}$, and construct the related 
Poisson-commutative subalgebra of $\gS(\q)$.

\subsection{Periodic gradings of Lie algebras}      
\label{subs:periodic}
Let $\vartheta\in\Aut(\q)$ be a Lie algebra automorphism of finite order $m\ge 2$ and $\zeta=\sqrt[m]1$ 
a primitive root of unity. Write also $\oth$ for the order of $\vartheta$. 
If $\q_i$ is the $\zeta^i$-eigenspace of $\vartheta$, $i\in \BZ_m$, then the direct sum
$\q=\bigoplus_{i\in \BZ_m}\q_i$ is a {\it periodic grading\/} or $\BZ_m$-{\it grading\/} of 
$\q$. The latter means that $[\q_i,\q_j]\subset \q_{i+j}$ for all $i,j\in \BZ_m$. Here $\q_0=\q^\vartheta$ 
is the fixed-point subalgebra for $\vartheta$ and each $\q_i$ is a $\q_0$-module. We will be primarily 
interested in periodic gradings of {\bf semisimple} 
Lie algebras, but such a general setting is going to be useful, too.

We choose $\{0,1,\dots, m-1\}\subset\BZ$ as a fixed set of representatives for $\BZ_m=\BZ/m\BZ$. 
Under this convention, we have
$\q=\q_0\oplus\q_1\oplus\ldots\oplus\q_{m-1}$ and
\beq   \label{eq:Z_m}
[\q_i,\q_j]\subset \begin{cases}  \q_{i+j}, &\text{ if } \ i+j\le m{-}1, \\
 \q_{i+j-m}, &\text{ if } \ i+j\ge m. \end{cases}
\eeq
This is needed below, when we consider $\BZ$-graded contractions of $\q$ associated with $\vartheta$. 

The presence of $\vartheta$ allows us to split the Lie--Poisson bracket on $\q^*$ into a sum of two 
compatible linear Poisson brackets, as follows.
Consider the polynomial representation $\vp\!:\bbk^\star\to {\rm GL}(\q)$ 
such that $\vp_s(x)=s^j x$ for $x\in \q_j$.
As in Section~\ref{subs:contr-&-inv}, this defines a family of linear Poisson brackets in $\gS(\q)$ 
parametrised by $s\in\bbk^\star$, see~\eqref{eq:fi_s}.

\begin{prop}         \label{prop:compat}   
For any $\vartheta\in\Aut(\q)$ of finite order $m$ and $\vp$ as above, we have
\begin{itemize}
\item[\sf (i)] \ The map $\vp_s^{-1}:\q\to \q$ provides an isomorphism between $(\q, [\,\,,\,])$ and 
$(\q, [\,\,,\,]_{(s)})$. In particular, the Poisson brackets $\{\ ,\ \}_{(s)}$, $s\in\bbk^\star$, are isomorphic;
\item[\sf (ii)] \ there is a limit\/ $\lim_{s\to 0}  \{x,y\}_{(s)}=:\{x,y\}_{0}$, which is a linear 
Poisson bracket in $\gS(\q)$;
\item[\sf (iii)] \ the difference $\{\,\,,\,\}-\{\,\,,\,\}_{0}=:\{\,\,,\,\}_{\infty}$ is a linear Poisson bracket on 
$\gS(\q)$ and \\  \hspace*{2ex}
%\centerline{
$\{\,\,,\,\}_{(s)}=\{\,\,,\,\}_{0}+s^m\{\,\,,\,\}_{\infty}$ \ for any $s\in\bbk^\star$.
\item[\sf (iv)] \ The Poisson bracket $\{\,\,,\,\}_{\infty}$ is obtained as the zero limit w.r.t. the polynomial  
representation $\psi\!:\bbk^\star \to {\rm GL}(\q)$ such that $\psi_s=s^m{\cdot}\vp_{s^{-1}}=
s^m{\cdot}\vp_s^{-1}$, $s\in\bbk^\star$. In other words, we have
$\{\,\,,\,\}_{\infty}=\lim_{s\to 0}\{\,\,,\,\}_{(\psi,s)}$.
\end{itemize}
\end{prop}
\begin{proof}  
{\sf (i)} 
This readily follows from Eq.~\eqref{eq:fi_s}. 

{\sf (ii)} \ If $x\in\q_i$ and $y\in\q_j$, then
$ \{x,y\}_{(s)}=\begin{cases} [x,y], & \text{ if }\  i+j<m; \\
 s^m [x,y],  & \text{ if }\  i+j\ge m.
\end{cases}$ 
\ \
Therefore, the limit of $\{x,y\}_{(s)}$ as $s$ tends to zero exists and is given by
$\{x,y\}_{0}:=\begin{cases} [x,y], & \text{ if }\  i+j<m \\
 0,  & \text{ if }\  i+j\ge m
\end{cases}$~. 
The limit of Poisson brackets is again a Poisson bracket, hence the Jacobi identity is satisfied for
$\{\,\,,\,\}_{0}$, cf.~Section~\ref{subs:contr-&-inv}. However, this is easily verified directly.

{\sf (iii)} \ By the above formula for $\{x,y\}_{0}$, we have
$\{x,y\}_{\infty}=\begin{cases} 0, & \text{ if }\  i+j<m; \\
[x,y],  & \text{ if }\  i+j\ge m.  \end{cases}$ \  
\\ Therefore $\{x,y\}_{(s)}=\{x,y\}_{0}+s^m\{x,y\}_{\infty}$ for all $x,y\in\q$.
It is also easily verified that $\{\,\,,\,\}_{\infty}$ satisfies the Jacobi identity.

{\sf (iv)} We have $\psi_s(x)=s^{m-i}x$ for $x\in\q_i$. Then an easy calculation shows that the $(\psi,s)$-bracket is given by 
$ \{x,y\}_{(\psi,s)}=\begin{cases} s^m [x,y], & \text{ if }\  i+j<m; \\
 [x,y],  & \text{ if }\  i+j\ge m.
\end{cases}$
\end{proof}

Two Poisson brackets on the algebra $\gS(\q)$ are said to be {\it compatible}, if any linear 
combination of them is again a Poisson bracket. Actually, if $\{\,\,,\,\}_{1}$ and $\{\,\,,\,\}_{2}$ are Poisson 
brackets, then it suffices to check that just $\{\,\,,\,\}_{1}+\{\,\,,\,\}_{2}$ is a Poisson 
bracket~\cite[Lemma\,1.1]{OY}. Anyway, we have

\begin{cl}   \label{cor:compatibel} 
For any $\vartheta\in \Aut(\q)$ of finite order,
the Poisson brackets $\{\,\,,\,\}_{0}$ and $\{\,\,,\,\}_{\infty}$ are compatible, and the corresponding pencil
contains the initial Lie--Poisson bracket.
\end{cl}

All Poisson brackets involved in Proposition~\ref{prop:compat} are linear. Therefore, in place of 
Poisson brackets on the symmetric algebra $\gS(\q)$, we can stick to the corresponding Lie algebra 
structures on the vector space $\q$. Let $\q_{(s)}$ be the Lie algebra corresponding to  
$\{\,\,,\ \}_{(s)}$. Then all algebras with $s\in\bbk^\star$ are isomorphic, whereas the brackets
$\{\,\,,\,\}_{0}$ and $\{\,\,,\,\}_{\infty}$ give rise to entirely different Lie algebras $\q_{(0)}$ and 
$\q_{(\infty)}$, respectively. Both $\q_{(0)}$ and $\q_{(\infty)}$ are Lie algebra contractions of $\q$ in the 
sense of \cite[Ch.\,7, \S\,2]{t41}. Therefore 
$\ind\q_{(0)}\ge \ind\q$ and $\ind\q_{(\infty)}\ge \ind\q$ (the semi-continuity of the index). 

\begin{prop}   \label{prop:graded-limits}
The Lie algebras $\q_{(0)}$ and $\q_{(\infty)}$ are $\BN$-graded. More precisely, if\/ $\rr[i]$ stands for the component of grade $i\in\BN$ in an $\BN$-graded Lie algebra $\rr$, then 
\[
    \q_{(0)}[i]=\begin{cases} \q_i \  & \text{ for } i=0,1,\dots,m{-}1 \\
       0 & \text{ otherwise}   \end{cases}, \ 
    \q_{(\infty)}[i]=\begin{cases} \q_{m-i} \  & \text{ for } i=1,2,\dots,m  \\
      0 & \text{ otherwise}   \end{cases}.
\]
In particular, $\q_{(\infty)}$ is nilpotent and the subspace $\q_0$, which is the 
highest grade component of $\q_{(\infty)}$, belongs to the centre of $\q_{(\infty)}$.
\end{prop}
\begin{proof}
Use the formulae for $\{x,y\}_{0}$ and $\{x,y\}_{\infty}$ from the proof of Proposition~\ref{prop:compat}. 
\end{proof}

\subsection{Poisson-commutative subalgebras related to compatible Poisson brackets}
\label{subs:construct-2}
There is a general method for constructing a Poisson-commutative subalgebra of $\gS(\q)$ with
``large" transcendence degree that exploits compatible Poisson brackets and the centre of $\gS(\q)$, see 
e.g.~\cite[Sect.~1.8.3]{duzu}, \cite[Sect.~10]{GZ}, \cite[Sect.~2]{OY}.  We recall this method in our present setting. 

By Proposition~\ref{prop:compat}, we have
\[
   \{\,\,,\,\}  = \{\,\,,\,\}_{0}+ \{\,\,,\,\}_{\infty} \ \text{ and } \  \{\,\,,\,\}_{(s)}  = \{\,\,,\,\}_{0}+ s^m\{\,\,,\,\}_{\infty}.
\]
Since $\{\,\,,\,\}_{(s)}=\{\,\,,\,\}_{(s')}$ if $s^m=(s')^m$, it is convenient to replace $s^m$ with $t$ and
set
\[
   \{\,\,,\,\}_{t}  = \{\,\,,\,\}_{0}+ t\{\,\,,\,\}_{\infty} , 
\]
where $t\in \BP:=\bbk\cup\{\infty\}$ and the value $t=\infty$ corresponds to the bracket
$\{\,\,,\,\}_{\infty}$. But, we will use the parameter $s\in\bbk^\star$, when the multiplicative group
$\vp: \bbk^\star\to \GL(\q)$ is needed. \\ \indent
Let $\q_{(t)}$ stand for the Lie algebra corresponding to $\{\,\,,\,\}_{t}$.
All these Lie algebras have the same underlying vector space. Since the algebras $\q_{(t)}$ 
with $t\in\bbk^\star$ are isomorphic, they have one and the same index. We say that $t\in\BP$ is 
{\it regular} if $\ind\q_{(t)}=\ind\q$ and write $\BP_{\sf reg}$ for the set of regular values.
Then $\BP_{\sf sing}:=\BP\setminus \BP_{\sf reg}\subset \{0,\infty\}$  is the set of singular values.

Let $\cz_t$ be the centre of the Poisson algebra $(\gS(\q),\{\,\,,\,\}_t)$. In particular, 
$\cz_1=\gS(\q)^{\q}$. If $t\in\bbk^\star$, then $\cz_t=\vp_s^{-1}(\cz_1)$, where  $s^m=t$. 
By Eq.~\eqref{eq:neravenstvo-ind}, we have $\trdeg\cz_t\le\ind\q_{(t)}$. Our main object is the 
subalgebra $\gZ\subset\gS(\q)$ generated by the centres $\cz_t$ with $t\in\BP_{\sf reg}$, i.e.,
\[
     \gZ=\gZ(\gt q,\vth)=\mathsf{alg}\langle\cz_t \mid t\in\BP_{\sf reg}\rangle .
\]
By a general property of compatible brackets, the algebra $\gZ$ is Poisson-commutative w.r.t. {\bf all} 
brackets $\{\,\,,\,\}_t$ with $t\in\BP$, cf. \cite[Sect.~2]{OY}. Note that the Lie subalgebra 
$\q_0\subset\q=\q_{(1)}$ is also the {\bf same} Lie subalgebra in any $\q_{(t)}$ with 
$t\ne\infty$ (cf. Proposition~\ref{prop:graded-limits} for $\q_{(0)}$). Therefore, 
\begin{equation}         \label{incl}
      \cz_t\subset\gS(\q)^{\q_0} \ \text{ for } \,t\ne\infty.
\end{equation}
In general, one cannot say much about $\gZ\subset \gS(\q)$. To arrive at more definite conclusions on 
$\gZ$, a lot of extra information on $\gS(\q)^\q$, $\q_{(0)}$, and $\q_{(\infty)}$ is required. In particular, 
one has to know whether $0$ and/or $\infty$ belong to $\BP_{\sf reg}$.
And this is the reason, why we have to stick to semisimple Lie algebras.

%%%%%%%%%%  Section 3   %%%%%%%%%%%%
\section{Poisson-commutative subalgebras of $\gS(\g)$: the semisimple case}
\label{sect:3} 

\noindent
From now on, $G$ is a connected semisimple algebraic group and $\g=\Lie G$.
We consider $\vth\in\Aut(\g)$ of order $m\ge 2$ and 
freely use the previous notation and results, with $\q$ being replaced by $\g$. In particular,
\[
  \g=\g_0\oplus\g_1\oplus\ldots\oplus \g_{m-1} ,
\] 
where $\{0,1,{\dots},m-1\}$ is the fixed set of representatives for $\BZ_m$, and $G_0$ is the connected 
subgroup of $G$ with $\Lie(G_0)=\g_0$. Then $\g_{(t)}$ is a family of Lie algebras parameterised by 
$t\in \BP=\bbk\cup\{\infty\}$, where the algebras $\g_{(t)}$ with $t\in\bbk^\star$ 
are isomorphic to $\g=\g_{(1)}$, while $\g_{(0)}$ and $\g_{(\infty)}$ are different $\BN$-graded 
contractions of $\g$. Next, $\cz_t$ is the Poisson centre of $(\gS(\g),\{\,\,,\,\}_t)$ and the construction of 
Section~\ref{subs:construct-2} provides a Poisson-commutative subalgebra 
$\gZ=\mathsf{alg}\langle\cz_t \mid t\in\BP_{\sf reg}\rangle\subset\gS(\g)$. The connected algebraic
group corresponding to $\g_{(t)}$ is denoted by $G_{(t)}$. 

Our goal is to demonstrate that there are many interesting cases, in which $\gZ$ is a polynomial algebra 
having the maximal possible transcendence degree. 
Let us recall standard invariant-theoretic properties of semisimple Lie algebras.

The Poisson centre $\gS(\g)^\g=\gS(\g)^G$ is a polynomial algebra of Krull dimension $l=\rk\g$ and 
$\ind\g=l$. Hence one has now the equality in Eq.~\eqref{eq:neravenstvo-ind}. Note also that $\g_0$ is
a reductive Lie algebra. Write $\cN$ for the cone of nilpotent elements of $\g$.
Let $\varkappa$ be the Killing form on $\g$. We identify $\g$ and $\g_0$ with 
their duals via $\varkappa$. Moreover, since $\varkappa(\g_i,\g_{j})=0$ if $i+j\not\in\{0,m\}$, the dual space of 
$\g_j$, $\g_j^*$, can be identified with $\g_{m-j}$.
Here $\g^*_{\sf reg}=\g_{\sf reg}=\{x\in\g\mid \dim\g^x=l\}$. 
By~\cite{ko63}, $\g$ has the {\sl codim}--$3$ property, i.e., $\mathrm{codim}_\g(\g\setminus \g_{\sf reg})=3$.
Recall also that $\cN\cap\g_{\sf reg}$ is non-empty, and it is a sole $G$-orbit, the {\it regular nilpotent orbit}.

{\bf Convention.} 
We think of $\g^*$ as the dual space for any Lie algebra $\g_{(t)}$
and sometimes omit the subscript `$(t)$' in $\g_{(t)}^*$. However, if $\xi\in\g^*$,
then the stabiliser of $\xi$ with respect to the coadjoint representation
of $\g_{(t)}$ is denoted by $\g_{(t)}^\xi$.
 
Each Lie algebra $\g_{(t)}$ has its own singular set 
$\g^*_{(t),\sf sing}=\g^*\setminus \g^*_{(t),\sf reg}$, which is regarded as a subset of $\g^*$.
If $\pi_t$ is the Poisson tensor of the bracket $\{\,\,,\,\}_t$, then 
\[
        \g^*_{(t),\sf sing}=\{\xi\in\g^* \mid \rk \pi_t(\xi)< \rk \pi_t\} ,
\]
which is the union of the coadjoint $G_{(t)}$-orbits in $\g^*$ having a non-maximal dimension.
For simplicity, we write $\g^*_{\infty,\sf sing}$ or $\g^*_{\infty,\sf reg}$ in place of 
$\g^*_{(\infty),\sf sing}$ or $\g^*_{(\infty),\sf reg}$.

\begin{prop}   \label{prop:codim-sing}  
The closure of\/ $\bigcup_{t\ne 0,\infty} \g^*_{(t),\sf sing}$ in $\g^*$ is a subset of codimension at 
least $2$.
\end{prop}
\begin{proof}
Let $\xi=\xi_0+\xi_1+\dots +\xi_{m-1}\in\g^*$, where $\xi_i\in\g_i^*$. Using Proposition~\ref{prop:compat}(i) and the dual representation $\vp^*:\bbk^\star\to \GL(\g^*)$, $s\mapsto \vp^*_s$, one readily verifies 
that $\xi\in \g^*_{\sf sing}=\g^*_{(1),\sf sing}$ if and only if $\vp^*_{s}(\xi)\in \g^*_{(s),\sf sing}$. Therefore, 
\[
   \bigcup_{t\ne 0,\infty} \g^*_{(t),\sf sing} =\bigcup_{s\in\bbk^\star} \vp^*_s(\g^*_{\sf sing})= 
   \{\xi_0+s^{-1} \xi_1 +\ldots + s^{1-m}\xi_{m-1} \mid \xi\in\g^*_{\sf sing}, \, s\in\bbk^\star\} .
\]
Since $\codim \g^*_{\sf sing}=3$, the closure of\/ $\bigcup_{t\ne 0,\infty} \g^*_{(t),\sf sing}$ is a subset 
of $\g^*$ of codimension at least $2$.
\end{proof}

In order to compute $\trdeg\gZ$, we have to elaborate on some relevant properties of
the limit Lie algebras $\g_{(\infty)}$ and $\g_{(0)}$. 

\subsection{Properties of $\g_{(\infty)}$}   
\label{subs:g-infty}
By Proposition~\ref{prop:graded-limits}, $\g_{(\infty)}$ is a nilpotent  $\BN$-graded Lie 
algebra, hence $G_{(\infty)}$ is a unipotent algebraic group. 
Recall also that the subspace $\g_0$ belongs to the centre of $\g_{(\infty)}$. 
\begin{thm}            \label{thm:ind-inf}
For any $\vartheta\in\Aut(\g)$ of finite order, one has $\ind\g_{(\infty)}= \dim\g_0+\rk\g-\rk\g_0$.
\end{thm}
\begin{proof}
{\bf (1)} \ For  $\mu\in\g_0^*\subset\g_{(\infty)}^*$, the stabiliser $\g^\mu_{(\infty)}$ is also $\BN$-graded. 
%$\varphi_s$-stable subspace of $\g$.
Furthermore, if $\eta\in\g_j$ with $1\le j\le m-1$, then
$\varkappa(\mu,[\eta,\g_{m-j}]_{\infty})=0$ if and only if $[\eta,\mu]=0$. Hence
\beq     \label{eq:stab-infty}
               \g^\mu_{(\infty)}=\g^\mu + \g_0,
\eeq
where $\g^\mu$ is the usual stabiliser of $\mu\in\g^*$ (w.r.t. the initial Lie algebra structure on $\g$).
As is well-known, the reductive subalgebra $\g_0=\g^\vartheta$ contains regular semisimple elements 
of $\g$, see  e.g.~\cite[\S 8.8]{kac}.
These elements form a dense open subset of $\g_0$, which is denoted by $\Omega_0$. 
If $h\in\Omega_0$, then
$\g^h$ is a Cartan subalgebra of $\g$ and $\g^h\cap\g_0$ is a Cartan subalgebra of $\g_0$.
Using the identification of $\g_0$ and $\g_0^*$, we may think of the subset $\Omega_0^*$ of
``regular semisimple'' elements of $\g_0^*\subset\g_{(\infty)}^*$. For $\mu\in\Omega_0^*$, it follows 
from~\eqref{eq:stab-infty} that $\dim \g^{\mu}_{(\infty)} = \dim\g_0+\rk\g-\rk\g_0$ and hence 
$\ind\g_{(\infty)}\le \dim\g_0+\rk\g-\rk\g_0$.

{\bf (2)} \ Let us prove the opposite inequality.
We think of $\g^*_{(\infty)}$ as a graded vector space of the form
\[
   \g^*_{(\infty)}=\g_{0}^*\oplus\g_{1}^*\oplus\ldots\oplus\g_{m-1}^* .
\]
Write $\ad_{(\infty)}^*$ for the coadjoint representation of $\g_{(\infty)}$.
The graded structure of $\g_{(\infty)}$ described in Proposition~\ref{prop:graded-limits} implies that
$\ad_{(\infty)}^*$ has the property that 
\beq     \label{eq:graded-infty}
   \ad^*_{(\infty)}(\g_j){\cdot}\g_{i}^*\subset \g_{i+m-j}^*. 
\eeq
Take any $\xi=\sum_{j=0}^{m-1} \xi_j\in \g^*_{(\infty)}$ such that $\xi_0\in\Omega_0^*$.
Let $h\in\Omega_0$ be the regular semisimple element of $\g$ corresponding to $\xi_0$ under our 
identifications. Set $\te=\g^h$. Then $[\g,h]=\te^{\perp}$ is the orthogonal complement of 
$\te$ with respect to $\varkappa$ and $\g=[\g,h]\oplus \te$. For $\te_j=\te\cap\g_j$, this implies that
$[\g_j,h]=\g_j\cap \te_{m-j}^{\perp}$ and $[\g_j,h]\oplus \te_j=\g_j$. Our goal is to prove that 
the orbit $G_{(\infty)}{\cdot}\xi$ contains an element $\gamma$ such that  
$ \dim\g_{(\infty)}^\gamma\ge \dim\g_0+\rk\g-\rk\g_0$. We perform this step by step, as follows.

For $\eta\in\g_{m-1}$, we have $\ad_{(\infty)}^*(\eta){\cdot}\xi_0\in\g^*_1$ and 
\[
  \exp(\ad_{(\infty)}^*(\eta)){\cdot}\xi\in \xi_0+(\xi_1+\ad_{(\infty)}^*(\eta){\cdot}\xi_0)+ (\bigoplus_{j\ge 2} \g_j^*) .
\] 
Since $\ad_{(\infty)}^*(\g_{m-1}){\cdot}\xi_0=\g_{1}^*\cap \Ann(\te_1)$, 
there is $\eta_1\in\g_{m-1}$ such that 
$\gamma_1:=\xi_1+\ad_{(\infty)}^*(\eta_1){\cdot}\xi_0 \in\te_1^*$. 
Next, applying $\exp(\ad_{(\infty)}^*(\eta))$ with $\eta\in\g_{m-2}$ to 
$\exp(\ad_{(\infty)}^*(\eta_1))\xi=\xi_0+\gamma_1+\xi'_2+\dots$, we do not affect the summands 
$ \xi_0+\gamma_1$. In doing so, we can replace $\xi_2'$ with $\gamma_2\in \te_2^*$.
Eventually, we get in $G_{(\infty)}{\cdot}\xi$ an element of the form 
$\gamma=\xi_0+\sum_{j=1}^{m-1}\gamma_j$ with $\gamma_j\in\te_j^*$. 
It is easily seen that $\te\subset \g_{(\infty)}^\gamma$, hence 
\[
    \dim\g_{(\infty)}^\gamma\ge \dim\g_0+\rk\g-\rk\g_0 ,
\]
as required.
\end{proof}

\begin{cl}  \label{cor:infty=reg}
One has $\infty\in \BP_{\sf reg}$ if and only if\/ $\dim\g_0=\rk\g_0$, i.e., $\g_0$ is an abelian
subalgebra of $\g$. 
\end{cl}

\begin{rmk}    \label{rem:short-proof}
(1) There is a short proof of the corollary that does not use Theorem~\ref{thm:ind-inf} in
full strength. If $\infty\in \BP_{\sf reg}$, then $\cz_\infty\subset \gZ$. Since $\g_0$ belongs  
to the centre of $\g_{(\infty)}$, we have $\g_0\subset\cz_\infty$. Hence $\g_0$ has to be abelian in $\g$. 
Conversely, if $\g_0$ is abelian, then, for $\mu\in\Omega_0^*$ in the first part of the proof, we obtain
$\dim\g^\mu_{(\infty)}=\rk\g$. Hence $\ind\g_{(\infty)}=\rk\g$ and $\infty\in \BP_{\sf reg}$.
\\ \indent
(2) By V.G.\,Kac's classification of elements of finite order in $G$~\cite[Chap.\,8]{kac}, if $\g$ is simple, 
$\vartheta$ is inner, and $\g_0=\g^\vartheta$ is abelian, then $\oth$ is at least the Coxeter number of 
$\g$. This means that, for many interesting examples  with small $\oth$, we have
$\infty\in \BP_{\sf sing}$.
\end{rmk}
Let us check another technical condition, which is required below. 
Recall from Section~\ref{subs:coadj} that $\pi(\gamma)=\hat\gamma$ is a 
skew-symmetric bilinear form on $\g$ associated with $\gamma\in\g^*$.
We then need the rank of the restriction
$\pi(\xi)|_V$, if $\xi\in\g^*$ is generic and $V=\ker\pi_\infty(\xi)=\g_{(\infty)}^\xi$.
Set $\g_{>0}=\bigoplus_{j>0}\g_j$ and $\g_{>0}^*=\bigoplus_{j>0}\g_j^*$.

\begin{lm}               \label{rank-ker}
For any $\vartheta$ and generic $\xi\in\g^*$, we have $\rk\!(\pi(\xi)|_V) = \dim V - \rk\g$.
\end{lm}
\begin{proof}
Take any $\xi\in\g^*_{\sf reg}\cap \g_{\infty,{\sf reg}}^*$. Then $\rk\!(\pi(\xi)|_V)\le \dim V - \rk\g$
by~\cite[Appendix]{OY} and $\dim V=\dim\g_0+\rk\g-\rk\g_0$ by Theorem~\ref{thm:ind-inf}.
Write $\xi=\xi_0+\xi'$ with $\xi_0\in\g_0^*$ and $\xi'\in\g_{>0}^*$. 
Assume also that $\xi_0\in(\g_0)^*_{\sf reg}$. Note that this is an open condition on $\xi$, too.

As $\g_{0}\subset V$ (cf. Eq.~\ref{eq:stab-infty}), one can write $V=\g_{0}\oplus V'$ for some 
$V'\subset \g_{>0}$. Here we have $\dim V'=\rk\g-\rk\g_0$.
Let $\widetilde{V}$  be the  kernel of $\pi(\xi)|_V$. Then $\widetilde{V}\cap\g_0\subset\g_0^{\xi_0}$
and 
\[
  \rk\g_0 \ge \dim (\widetilde{V}\cap\g_0)\ge \dim\widetilde{V}-\dim V'=\dim\widetilde{V}+\rk\g_0-\rk\g .
\]
Hence $\dim\widetilde{V}\le \rk\g$ and thereby $\rk(\pi(\xi)|_V)\ge  \dim V-\rk\g$. 
This settles the claim. 
\end{proof}

For future use, we record yet another property of $\g_{(\infty)}$.

\begin{lm}             \label{lm:xi}
If $\xi=\xi_0+\xi'$, where  $\xi'\in \g_{>0}^*$,  and $\xi_0\in \g_0^*$ is regular in $\g$, then $\xi\in \g_{\infty,{\sf reg}}^*$.
\end{lm}
\begin{proof} As in Proposition~\ref{prop:compat}{\sf (iv)}, consider the invertible linear map
$\psi_s :\g \to \g$ such that $\psi(x_j)=s^{m-j} x_j$ for $x_j\in\g_j$  and $s\in\bbk^\star$.
It follows from Proposition~\ref{prop:graded-limits} that $\psi_s$ 
is an automorphism of the graded Lie algebra $\g_{(\infty)}$. 
Let $\psi_s^*$ denote the induced action on $\g^*_{(\infty)}$, i.e.,
$\psi_s^*(\xi_j)=s^{j-m}\xi_j$ if $\xi_j\in\g_j^*$. 
Then $\dim\g^\xi_{(\infty)}=\dim\g^{\psi_s^*(\xi)}_{(\infty)}$ for any $\xi\in\g^*$. Note that 
$\lim_{s\to 0} s^m \psi_s(\xi)=\xi_0$. If $\xi_0$ is regular in $\g_{(\infty)}^*$, then so is $\xi$. 
It remains to deal with $\xi_0$.  

In view of \eqref{eq:stab-infty}, $\dim\g^{\xi_0}_{(\infty)}=\dim\g_0+\dim\g^{\xi_0}-\dim\g_0^{\xi_0}$. 
Because $\xi_0$ is regular in $\g$, it is also regular in $\g_0$. 
Thus $\dim\g^{\xi_0}_{(\infty)}=\ind\g_{(\infty)}$, and we are done. 
\end{proof}

\subsection{Pencils of skew-symmetric matrices and differentials}  
\label{subs:pssm} 
Recall that any $F\in \gS(\g)$ is regarded as a polynomial function on $\g^*$.
Then $\textsl{d} F$ is the differential of $F$, which is a polynomial mapping from $\g^*$ to $\g$.
If $\gamma\in\g^*$, then $\textsl{d}_\gamma F\in \g$ stands for the value of $\textsl{d} F$ at $\gamma$.

For a subalgebra $\ca\subset\gS(\g)$ and $\gamma\in\g^*$, set
$\textsl{d}_\gamma \ca=\langle \textsl{d}_\gamma F\mid F\in \ca\rangle_{\bbk}$.
Then $\trdeg \ca=\max_{\gamma\in\q^*} \dim \textsl{d}_\gamma \ca$. 
By the definition of $\gZ$, we have $\textsl{d}_\gamma \gZ=\sum_{t\in \BP_{\sf reg}}  \textsl{d}_\gamma \cz_t$. 
Since $\cz_t$ is the centre of $(\gS(\g),\{\,\,,\,\}_{t})$, there is the inclusion 
 $\textsl{d}_\gamma \cz_t \subset \ker \pi_t(\gamma)$ for each $t\in\BP$ and each $\gamma\in\g^*$.   
 
Let $\{H_1,\dots,H_l\}$, $l=\rk\g$,  be a set of homogeneous algebraically independent generators of
$\gS(\g)^\g$. For any $H\in\gS(\g)^\g$, we have $\textsl{d}_\xi H\in\g^\xi$.
The {\it Kostant regularity criterion\/} for $\g$,  see \cite[Theorem~9]{ko63}, asserts that
\beq     \label{eq:ko-re-cr}
\text{ $\langle\textsl{d}_\xi H_j \mid 1\le j\le l\rangle_{\bbk}=\g^\xi$ \ if and only if \ $\xi\in\g^*_{\sf reg}$.}
\eeq
Clearly, this criterion applies to semisimple algebras $\g_{(t)}$ when $t\ne 0,\infty$. 
That is, 
\beq          \label{span-dif}
\text{if }\ t\ne0,\infty,  \text{ then }\  \xi \in\g^*_{(t),\sf reg} \ 
\Leftrightarrow \  
\textsl{d}_\xi \cz_t =\ker \pi_t(\xi) \Leftrightarrow \ \dim \ker \pi_t(\xi)=\rk\g . 
\eeq
As a step towards describing $\gZ$, we first consider the smaller algebra 
\[
      \gZ_{\times}=\mathsf{alg}\langle\cz_t \mid t\in\bbk^{\star}\rangle\subset\gZ .
\]
An open subset of an algebraic variety is said to be {\it big}, if its complement does not contain divisors. 
By Proposition~\ref{prop:codim-sing}, there is a big open subset $U_{\sf sr}\subset\g^*$ such that 
$\xi\in \g_{(t),{\sf reg}}^*$ for any $\xi\in U_{\sf sr}$ and any $t\ne 0,\infty$. 
Suppose that $\xi\in U_{\sf sr}$. Then 
\beq         \label{sum-dif}
\textsl{d}_\xi \gZ_{\times}=\sum_{t\in\bbk^{\star}} \ker\pi_t(\xi) \subset 
 \sum_{t:\, \rk\pi_t(\xi)=l} \ker\pi_t(\xi)=:L(\xi).
\eeq
Let us recall a method that provides an upper bound on $\dim L(\xi)$ and thereby on $\trdeg\gZ_\times$,
see~\cite[Section~1.1]{kruks}.   

Let  $\mathcal P(\xi)=\{ c\pi_t(\xi)\mid c\in\mK, t\in\BP\}$ be a pencil of skew-symmetric $2$-forms on 
$\g$, which is spanned by $\pi(\xi)$ and $\pi_0(\xi)$.  A $2$-form in this pencil is said to be {\it regular\/} 
if $\rk(c\pi_t(\xi))=l$. Otherwise, it is {\it singular}. Set 
$U_{\sf srr}=U_{\sf sr}\cap\g_{(0),{\sf reg}}^*\cap \g_{\infty,{\sf reg}}^*$. It is a dense open subset 
of $\g^*$, which may not be big. 

\begin{prop}       \label{prop:krest}
If  $\xi\in U_{\sf srr}$, then $\textsl{d}_\xi\gZ_{\times}=\textsl{d}_\xi\gZ=L(\xi)$.
\end{prop}
\begin{proof}
Suppose that $\xi\in U_{\sf sr}$. The space $L(\xi)$ is the sum of kernels over the regular lines in the pencil.  Since 
$\BP_{\sf reg}\setminus\bbk^\star$ is finite, we have  $\sum_{t\in\bbk^\star}\ker\pi_t(\xi)=L(\xi)$
by~\cite[Appendix]{codim3}. Hence $\textsl{d}_\xi \gZ_{\times}=L(\xi)$.  
If  $\xi\in U_{\sf srr}$, then $\textsl{d}_\xi\gZ\subset L(\xi)$.  Thereby 
$\textsl{d}_\xi\gZ_{\times}\subset \textsl{d}_\xi\gZ \subset L(\xi)$, which  leads to the equality $\textsl{d}_\xi\gZ_{\times}=\textsl{d}_\xi\gZ$. 
\end{proof}

\begin{cl}           \label{cl-4}
For a generic $\xi\in U_{\sf srr}$, we have $\dim L(\xi)=\trdeg\gZ_{\times}=\trdeg\gZ$.
In particular, $\gZ_{\times}\subset\gZ$ is an algebraic extension.
\end{cl}
Although three dimensions in Corollary~\ref{cl-4} are equal, we do not  know yet their exact value.  
For a given $\xi\in U_{\sf sr}$, $\dim L(\xi)$ depends on the number of singular lines in $\mathcal P(\xi)$.
 
\begin{prop}[{\cite[Proposition~1.5]{kruks}}]       \label{P} 
If $\xi\in U_{\sf sr}$, then   
\begin{itemize}
\item[{\sf (i)}] $\dim L(\xi)=\bb(\g)$ if and only if all  nonzero $2$-forms in $\mathcal P(\xi)$ are regular; 
\item[{\sf (ii)}] if \ $\mathcal P(\xi)$  contains a singular line, say $\bbk \pi_{t'}(\xi)$, 
then $\dim L(\xi)\le l+\frac{1}{2}\rk\pi_{t'}(\xi)$; 
\item[{\sf (iii)}] furthermore, $\dim L(\xi) = l+\frac{1}{2}\rk\pi_{t'}(\xi)$  if and only if\/ 
\begin{itemize}
\item \ $\bbk \pi_{t'}(\xi)$ is the 
only singular line in $\mathcal P(\xi)$, and 
\item \ $\rk(\pi(\xi)|_V)=\dim V-l$ for $V=\ker\pi_{t'}(\xi)$.  
\end{itemize}\end{itemize}
\end{prop} 

\noindent
By Corollary~\ref{cl-4}, $\gZ$ is algebraic over $\gZ_{\times}$, and
it follows from~\eqref{incl} that $\gZ_{\times}\subset\gS(\g)^{\g_0}$. 
Since, being an algebra of invariants, $\gS(\g)^{\g_0}$ is an algebraically closed
subalgebra of $\gS(\g)$, we conclude that
\beq        \label{eq:Z-v-inv-g_0}
            \gZ\subset\gS(\g)^{\g_0} .
\eeq 
By~\cite[Prop.\,1.1]{m-y}, if $\ca$ is a Poisson-commutative subalgebra
of $\gS(\g)^{\g_0}$, then 
\beq        \label{eq:MY19}
   \trdeg\ca\le \bb(\g)-\bb(\g_0)+\rk\g_0=:\bb(\g,\vartheta) . 
\eeq
Note that
$\bb(\g,\vth)\le \bb(\g)$ and the equality occurs if and only if $\g_0$ is abelian.
The most interesting case is that in which the upper bound in \eqref{eq:MY19}
is attained for $\ca=\gZ$.

\begin{thm}       \label{thm-trdeg}
If $\vartheta$ has the property that $\ind\g_{(0)}=\ind\g$, then $\trdeg\gZ=\bb(\g,\vartheta)$. 
\end{thm}
\begin{proof}
Suppose that $\xi\in U_{\sf srr}$. Then 
$\rk\pi_t(\xi)=\dim\g-l$ if and only if $t\in\BP_{\sf reg}$. 
Since $\trdeg\gZ\le\bb(\g,\vartheta)$, it suffices to show that $\dim L(\xi)=\bb(\g,\vartheta)$ whenever 
$\xi\in U_{\sf srr}$ is generic. 

Suppose first that $\infty\in\BP_{\sf reg}$. Then $\rk\pi_t(\xi)=\dim\g-l$ for 
any $t\in\BP$ and,
by Proposition~\ref{P}{\sf (i)}, we obtain $\dim L(\xi)=\bb(\g)$. 
And the equality $\ind\g=\ind\g_{(\infty)}$ means that $\g_0$ is abelian, hence
$\bb(\g)=\bb(\g,\vartheta)$.

Suppose  that $\infty\not\in\BP_{\sf reg}$. Then $\pi_{\infty}(\xi)$ spans the unique singular line in the 
pencil $c\pi_t(\xi)$. By Lemma~\ref{rank-ker}, there is a dense open subset $U\subset \g^*$  such that,
for any $\eta\in U$ and $V=\ker\pi_{\infty}(\eta)$, we have  
$\rk(\pi(\eta)|_V)=\dim V-l$. If $\xi\in U_{\sf srr}\cap U$, then conditions~{\sf (ii)} and {\sf (iii)} of Proposition~\ref{P}
are satisfied for $\pi_t(\xi)$. Thus, in this case, 
\[
\dim L(\xi)=l+\frac{1}{2}\rk\pi_\infty=l+\frac{1}{2}(\dim\g-\dim\g_0-l+\rk\g_0)=
\bb(\g)-\bb(\g_0)+\rk\g_0=\bb(\g,\vartheta),
\]
since $\rk\pi_\infty(\xi)=\rk\pi_\infty=\dim\g-\ind\g_{(\infty)}$ and $\ind\g_{(\infty)}$ is computed in
Theorem~\ref{thm:ind-inf}. 
\end{proof}

\subsection{Properties of $\g_{(0)}$}   \label{subs:g-null}
Study of Lie algebras $\g_{(0)}$ associated with arbitrary periodic automorphisms of $\g$ have been initiated in~\cite{p09}, where they are called {\it cyclic contractions\/} or $\BZ_{k}$-{\it contractions}.
In~\cite{p09}, these algebras are denoted by $\g\langle k\rangle_0$, where $k=\oth$, because of their
interpretation as the fixed-point subalgebras of an extension of $\vth$ to an automorphism of $k$-th 
Takiff algebra modelled on $\g$. 
(We discuss Takiff algebras in Section~\ref{sect:cyc}.)

The case in which $\oth\ge 3$ appears to be more difficult than that of involutions. If $\oth=2$, then 
$\ind\g_{(0)}=\ind\g$ and $\g_{(0)}$ has the {\sl codim}--$2$ property~\cite{p07}. Whereas, 
for $\oth\ge 3$, it can happen that $\g_{(0)}$ does not have the {\sl codim}--$2$ property, and the
equality $\ind\g_{(0)}=\ind\g$ is only known under certain constraints.

To state some sufficient conditions, we first recall some results of {E.B.\,Vinberg}~\cite{vi76}.
Associated with $\BZ_m$-grading $\g=\bigoplus_{i\in\BZ_m}\g_i$, one has the linear action of $G_0$
on $\g_1$. Then 
\begin{itemize}
\item the algebra of invariants $\bbk[\g_1]^{G_0}$ is a polynomial (free) algebra;
\item the morphism $\pi : \g_1 \to \g_1\md G_0=\spe(\bbk[\g_1]^{G_0})$ is flat and surjective;
\item $\pi^{-1}\pi(0)=\cN\cap\g_1$ and each fibre of $\pi$ contains finitely many $G_0$-orbits. 
\end{itemize}
(It is worth mentioning that this is only a tiny fraction of fundamental results obtained in that great 
paper.) The fibre $\pi^{-1}\pi(0)$ 
is customary called the {\it null-cone\/} (w.r.t. the $G_0$-action on $\g_1$) and we denote it by $\cN_1$.
\begin{df}     \label{df:N-reg}
Following~\cite{p09}, we say that 

(1) \ $\vth$ is $\gS$-{\it regular}, if $\g_1$ contains a regular semisimple element of $\g$;
 
(2) \ $\vth$ is $\gN$-{\it regular}, if $\g_1$ (i.e., $\cN_1$) contains a regular nilpotent element of $\g$;

(3) \ $\vth$ is {\it very $\gN$-regular}, if each irreducible component of $\cN_1$ contains a regular nilpotent element of $\g$.
\end{df}
If $\oth=2$, then it follows from~\cite{kr71} that properties (1)--(3) are equivalent.  But this is not always 
the case if $\oth\ge 3$. It can happen that $\vth$ is $\gS$-{regular}, but not $\gN$-{regular}; and vice 
versa. It can also happen that $\cN_1$ is reducible and some irreducible components of $\cN_1$ 
are not reduced (in the scheme-theoretic sense). Examples of automorphisms of order ${\ge} 3$ such 
that good properties of Definition~\ref{df:N-reg} hold are given in Examples~5.9 and 5.10 in~\cite{p09}.
The following assertion provides sufficient conditions for some good properties of $\g_{(0)}$ to hold.

\begin{thm}[{\cite[Sect.\,5]{p09}}]     \label{thm:from-09}
Let $\vth\in\mathsf{Aut}(\g)$ be of finite order. 
\begin{itemize}
\item If\/ $\g_1\cap\g_{\sf reg}\ne\varnothing$, then $\ind\g_{(0)}=\ind\g$;
\item If\/ $\vth$ is both $\gS$-{regular} and {very $\gN$-regular}, then
 \begin{itemize}
 \item $\g_{(0)}$ has the {\sl codim}--$2$-property;
 \item $\cz_0$ is 
  is freely generated by $H^\bullet_1,\dots,H^\bullet_l$, $l=\rk\g$, where $\{H_1,\dots,H_l\}$ is {\bf any} set of homogeneous generators of\/ $\gS(\g)^\g=\cz_1$ that consists of $\vth$-eigenvectors. 
  \end{itemize}
\end{itemize}
\end{thm}
\noindent

The assumptions of Theorem~\ref{thm:from-09} are not always satisfied, and
the property of being ``very $\gN$-regular'' is difficult to check directly. Some methods for handling these
properties and related examples can be found in Section~5 in~\cite{p09}.
\\ \indent
There is a nice special case, where all properties of Definition~\ref{df:N-reg} hold and
Theorem~\ref{thm:from-09} applies. Namely, let $\g$ be the direct sum of $m$ copies of a semisimple
Lie algebra $\h$ and $\vth$ a cyclic permutation of the summands; hence $\oth=m$. Here we obtain a 
complete description of the Poisson-commutative subalgebra $\gZ$, see Section~\ref{sect:cyc}.

%%%%%%%%%%%    Section  %%%%%%%%%%%%%%%
\section{On algebraically independent generators} 
\label{sect:free}

\noindent
If it is known that $\trdeg\gZ=\bb(\g,\vth)$, then it becomes a meaningful task to compute the minimal number of generators of $\gZ$ or $\gZ_\times$.

Let $\{H_1,\dots,H_l\}$  be a set of homogeneous algebraically independent generators of
$\gS(\g)^\g$ and $d_i=\deg H_i$. Then $\sum_{i=1}^l d_i=\bb(\g)$.  Recall from 
Sections~\ref{subs:contr-&-inv} and \ref{subs:periodic} that associated with $\vth$, we have the
polynomial homomorphism $\vp: \bbk^\star\to {\rm GL}(\g)$ and its extension to invertible linear transformations of $\gS^j(\g)$ for all $j$. Therefore, each $H_j$ decomposes as 
$H_j=\sum_{i\ge 0} H_{j,i}$, where $\varphi_s(H_j)=\sum_{i\ge 0} s^i H_{j,i}$. The polynomials
$H_{j,i}$ are called {bi-homogeneous components} of $H_j$. By definition, the $\vp$-{\it degree\/} of
$H_{j,i}$ is $i$, also denoted by $\deg_\vp H_{j,i}$.
\\ \indent
Let $H_j^\bullet$ be the nonzero bi-homogeneous component of $H_j$ with
maximal $\vp$-degree. Then $\deg_{\vp}\! H_j=\deg_{\vp}\! H_j^\bullet$ and we set
$d_j^\bullet=\deg_{\vp}\! H_j^\bullet$.
By Proposition~\ref{prop:compat}{\sf (i)}, $\cz_{t}$ is generated by 
$\vp_s^{-1}(H_1),\dots,\vp_s^{-1}(H_l)$, where $t=s^m\in\bbk^\star$. 
By the standard argument with the Vandermonde determinant, we then conclude that 
$\gZ_{\times}$ is generated by all bi-homogeneous components of $H_1,\dots, H_l$, i.e.,
\beq     \label{eq:alg-kr}
               \gZ_{\times}=\mathsf{alg}\langle H_{j,i} \mid 1\le j\le l, 0\le i \le d_j^\bullet \rangle.
\eeq

\begin{df}
Let us say that $H_1,\dots,H_l$ is a {\it good generating system} in $\gS(\g)^\g$
({\sf g.g.s.}\/ {\it for short}) for $\vartheta$, if $H_1^\bullet,\dots,H_l^\bullet$ are
algebraically independent. Then we also say that $\vartheta$ {\it admits\/} a {g.g.s.}
\end{df}
\noindent
The property of being `good' really depends on a generating system. For instance, for $\g=\sln$ and 
$m=2$,  the coefficients of the characteristic polynomial of $A\in\sln$ yield a {\sf g.g.s.}, while the 
polynomials $\tr(A^i)$, $i=2,\dots,n$ do not provide a {\sf g.g.s.}, see \cite{p07}.

The importance of {\sf g.g.s.} is manifestly seen in the following fundamental result.

\begin{thm}[{\cite[Theorem\,3.8]{contr}}]    \label{thm:kot14}
Let $H_1,\dots,H_l$ be an arbitrary set of homogeneous algebraically independent generators of\/ $\gS(\g)^\g$. Then
\begin{itemize}
\item[\sf (i)] \ $\sum_{j=1}^l \deg_{\vp}\! H_j\ge \sum_{j=1}^{m-1} j\dim\g_j=:D_\vartheta$\,;
\item[\sf (ii)] \  $H_1,\dots,H_l$ is a {\sf g.g.s.} if and only if\/ $\sum_{j=1}^l \deg_{\vp}\! H_j=D_\vartheta$\,;
\item[\sf (iii)] \  if\/ $\g_{(0)}$ has the {\it codim}--$2$ property, $\ind\g_{(0)}=l$, and  $H_1,\dots,H_l$ is a {\sf g.g.s.}, then
$\cz_0=\gS(\g_{(0)})^{\g_{(0)}}$ is a polynomial algebra freely generated by 
$H_1^\bullet,\dots,H_l^\bullet$. 
\end{itemize}
\end{thm}

Recall that  $\dim\g_j=\dim\g_{m-j}$ for any $0 < j\le m-1$.  Hence
\[
    D_\vartheta=\sum_{j=1}^{m-1} j\dim\g_j=\frac{1}{2}\sum_{j=1}^{m-1} m \dim\g_j = \frac{m}{2} (\dim\g-\dim\g_0).
\]
Not every $i\in \{0,1,\dots, d_j^\bullet\}$ provides a {\bf nonzero} bi-homogeneous
component $H_{j,i}$. To make this precise, we first consider the case of inner automorphisms, which is technically easier.

\begin{thm}   \label{thm:main3-2}
Suppose that $\vartheta\in {\sf Aut}(\g)$ is inner and admits a {\sf g.g.s.} Let $H_1,\dots,H_l$ be such a {\sf g.g.s.} in $\gS(\g)^\g$ and
$d_j^\bullet=\deg_{\vp}\! H_j$. Then
\begin{itemize}
\item[\sf (i)]  each $d^\bullet_j\in m\BZ$ \ and $H_{j,i}\ne 0$ only if\/ $i\in m\Z$ and $0\le i\le d_j^\bullet$;
\item[\sf (ii)]  if\/ $\ind\g_{(0)}=l$, then
$H_{j,i}\ne 0$ for every $i\in m\Z$ with $0\le i\le d_j^\bullet$ and
the polynomials $\{H_{j,i} \mid j=1,\dots,l; \ \& \ i=0,m,\dots, d^\bullet_j \}$ are algebraically independent 
(and generate $\gZ_{\times}$).
\end{itemize}
\end{thm}
\begin{proof}
{\sf (i)} Since $\vartheta$ is inner, $\rk\g=\rk\g_0$ and $\vartheta(H_j)=H_j$ for all $j$. 
On the other hand, $\vartheta(H_{j,i})=\zeta^{i}H_{j,i}$. This proves everything.

{\sf (ii)} It follows from {\sf (i)} that the number of nonzero bi-homogeneous components of $H_j$ is at
most $(d_j^\bullet/m) +1$. Hence the total number of nonzero bi-homogeneous components of all $H_j$
is at most
\[
   \sum_{j=1}^l  \left(\frac{d_j^\bullet}{m} +1\right)=\frac{D_\vartheta}{m}+l=
   \frac{1}{2}(\dim\g-\dim\g_0)+l=\bb(\g)-\bb(\g_0)+\rk\g_0=\bb(\g,\vth),   
\]
where the equality $\rk\g=\rk\g_0$ is used. On the other hand,
the hypothesis $\ind\g_{(0)}=l$ guarantee us that $\trdeg\gZ_\times=\trdeg\gZ_{\times}=
\bb(\g,\vth)$,
cf. Theorem~\ref{thm-trdeg} and Corollary~\ref{cl-4}. As
the bi-homogeneous components $\{H_{j,i}\}$ generate $\gZ_{\times}$~\eqref{eq:alg-kr},   
all $H_{j,i}$ with $i/m\in\BZ$ must be nonzero and algebraically independent.
\end{proof}

With extra technicalities, the same idea works for the arbitrary automorphisms as well. Let 
$\vth\in\Aut(\g)$ be an arbitrary automorphism of order $m$. Since $\vartheta$ acts on $\gS(\g)^{\g}$, 
there is a set of homogeneous generators $\{H_j\}\subset\gS(\g)^\g$ such that each $H_j$ is an 
eigenvector of $\vartheta$, i.e., $\vth(H_j)=\zeta^{r_j} H_j$ for some $r_j\in\BZ$.  However, we need a 
set of free generators that simultaneously is a {\sf g.g.s}{.} and consists of $\vartheta$-eigenvectors.
For $m=2$, this is proved in~\cite[Lemma\,3.4]{OY}. The following is an adaptation of that argument
to any $m\ge 2$.

\begin{lm}        \label{kosiks-lemma}
If $\vth$ admits a {\sf g.g.s.}, then there is also a {\sf g.g.s.} that consists of $\vth$-eigenvectors.
\end{lm}
\begin{proof}
Let $H_1,\dots,H_l$ be a {\sf g.g.s.}, hence $\sum_{j=1}^l \deg_\vp H_j=D_\vartheta$ in view 
of Theorem~\ref{thm:kot14}.

Let $\gA_+$ be the ideal in $\gS(\g)^\g$ generated by all homogeneous invariants of positive degree.
Then $\boldsymbol{\gA}:=\gA_+/\gA_+^2$ is a finite-dimensional $\mK$-vector space. If $H\in \gA_+$,
then $\bar H:=H+\gA_+^2\in \boldsymbol{\gA}$.
As is well-known, $F_1,\dots,F_m$ is a generating system for $\gS(\g)^\g$ if and only if
the $\mK$-linear span of $\bar F_1,\dots, \bar F_m$ is the whole of $\boldsymbol{\gA}$.
In our situation, $\dim_\mK \boldsymbol{\gA}=l$ and $\boldsymbol{\gA}=\langle
\bar H_1,\dots, \bar H_l\rangle_{\mK}$.

If $H_i$ is not a $\vartheta$-eigenvector, i.e., $\vartheta(H_i)\not\in\mK H_i$, then we consider the 
(non-minimal) generating set
\[
\{   H_1,\dots,H_{i-1},  H_{i+1},\dots, H_l \} \ \bigcup \ \left\{ H_i^{[k]}=\sum_{j=0}^{m-1} \zeta^{jk}\vartheta^j(H_i) \mid 0\le k<m \right\}
\]
for  $\gS(\g)^\g$ that includes $l+m-1$ polynomials. Since
$\bar H_1,\dots,\bar H_{i-1}, \bar H_{i+1},\dots, \bar H_l$ are linearly independent in $\boldsymbol{\gA}$,
one can pick up a suitable $H_i^{[k]}$ such that
$\bar H_1,\dots,\bar H_{i-1}, H_i^{[k]}, \bar H_{i+1},\dots, \bar H_l$ is again a minimal generating set.
Let us demonstrate that there is actually a unique suitable choice of $H_i^{[k]}$, and this yields a {\sf g.g.s.} as well. 
Recall that $d_j^\bullet=\deg_{\vp}\! H_j^\bullet=\deg_{\vp}\! H_j$.

Clearly $\deg_\vp  H_i^{[k]} \le d_i^\bullet$ for each $k$. Since $\sum_{j=0}^{m-1}\zeta^{j k'}=0$ if 
$1\le k'\le m-1$, we see that $\deg_\vp  H_i^{[k]} = d_i^\bullet$ if $(k+d_i^\bullet)\in m\Z$ and 
$\deg_\vp  H_i^{[k]} < d_i^\bullet$ for all other $k$. If $\deg_\vp  H_i^{[k]} < d_i^\bullet$, then the sum of 
$\vp$-degrees for $H_1,\dots,H_{i-1},H_i^{[k]}, H_{i+1},\dots, H_l$ is less than $D_\vartheta$. By 
Theorem~\ref{thm:kot14}, this means that such choice of $H_i^{[k]}$ in place of $H_i$ provides a set of 
algebraically dependent polynomials. Thus, the only right choice is to take $k$ such that 
$k+d_i^\bullet \equiv 0 \pmod m$, when Theorem~\ref{thm:kot14} also guarantee us that we obtain a 
{\sf g.g.s.}

The procedure reduces the number of generators that are not $\vartheta$-eigenvectors, and we 
eventually obtain a {\sf g.g.s.} that consists of $\vartheta$-eigenvectors.
\end{proof}

Let  $\{H_1,\dots,H_l\}\subset\gS(\g)^{\g}$  be a generating set consisting of $\vartheta$-eigenvectors.
Then  $\vartheta(H_i)=\zeta^{r_i} H_i$ with $0\le r_i<m$. The integers $\{r_i\}$ depend only on the
connected component of $\Aut(\g)$ that contains $\vth$, and if $a$ is the 
order of $\vth$ in ${\Aut}(\g)/\mathsf{Int}(\g)$, then $\zeta^{ar_i}=1$. Therefore, if $\g$ is simple, then 
$\zeta^{r_i} = \pm 1$ for all types but $\GR{D}{4}$.

\begin{lm}    \label{lm:outer-inv}
For any  $\vartheta\in {\sf Aut}(\g)$ of order $m$, we have

{\sf (1)} \ $\vartheta(H_j)=H_j$ if and only if $d^\bullet_j\in m\Z$;

{\sf (2)} \ $\sum_{j=1}^l r_j = \frac{1}{2}m(\rk\g-\rk\g_0)$;

{\sf (3)} \ $\rk\g_0=\# \{j \mid \vartheta(H_j)=H_j\}$.
\end{lm}
\begin{proof}
{\sf (1)} \ The proof is similar to that of Theorem~\ref{thm:main3-2}(i).
\\
{\sf (2)} \ Recall that $\g_0$ always contains regular semisimple elements of $\g$~\cite[\S 8.8]{kac}.
Therefore, if $\te_0$ is a Cartan subalgebra of $\g_0$, then $\te=\z_\g(\te_0)$ is a $\vartheta$-stable
Cartan subalgebra of $\g$. Since $\varkappa\vert_\te$ is non-degenerate,  
$\dim(\te\cap\g_j)=\dim(\te\cap\g_{m-j})$ for all $j<m$. 

Let us apply Kostant's regularity criterion (cf.~\eqref{eq:ko-re-cr})
to some $x\in\te_0\cap \g_{\sf reg}$. According to this criterion,
$\lg \textsl{d}_x H_j\mid 1\le j\le l\rg_{\bbk}=\te$. Here $\textsl{d}_x H_j\in \te\cap\g_i$ if and only if $r_j=i$.
Therefore
\[
     \sum_{j=1}^l r_j = \sum_{i=1}^{m-1} i \dim(\te\cap\g_i)=\frac{1}{2}m(\dim\te-\dim\te_0).
\]
%\noindent
{\sf (3)} To prove this, it suffices to notice that
$\textsl{d}_x H_j\in \g_0$ if and only if $\vartheta(H_j)=H_j$.
\end{proof}

Now, we can prove our main result on $\gZ_\times$.
\begin{thm}                          \label{free-main}
Suppose that $\vth\in\Aut(\g)$ admits a {\sf g.g.s.} and\/ $\ind\g_{(0)}=\rk\g$. Then 
\\ \indent
{\sf (i)} $\gZ_{\times}$ is a polynomial Poisson-commutative subalgebra of\/ $\gS(\g)^{\g_0}$ having 
the maximal transcendence degree. 
\\ \indent
{\sf (ii)} More precisely,
if $H_1,\dots,H_l$ is a {\sf g.g.s.} that consists of $\vth$-eigenvectors, then
$\gZ_{\times}$ is freely generated by the nonzero bi-homogeneous components of all $\{H_j\}$.
\end{thm}
\begin{proof}
We already know that $\gZ_{\times}$ is Poisson-commutative and 
$\trdeg\gZ_\times=\bb(\g,\vartheta)$ is maximal possible (Theorem~\ref{thm-trdeg}). Hence we have to 
prove the polynomiality and (ii).

Recall that $\{H_j\}$ is a {\sf g.g.s.} if and only if 
$\sum_{j=1}^l d_j^\bullet=D_\vth=m(\dim\g-\dim\g_0)/2$.
If $\vth(H_j)=H_j$, i.e., $d_j^\bullet\in m\Z$, then $H_j$ has at most $(d_j^\bullet/m)+1$ nonzero
bi-homogeneous components, as in the proof of Theorem~\ref{thm:main3-2}. 
In general,  if $\vth(H_j)=\zeta^{r_j} H_j$, then $H_{j,i}$ can be nonzero only if 
$i \equiv r_j \pmod m$. Therefore, $d_j^\bullet \equiv r_j \pmod m$ and
$H_j$ has at most $1+\frac{d_j^\bullet-r_j}{m}$ nonzero bi-homogeneous components. Using  Lemma~\ref{lm:outer-inv}, we see that
the total number of all nonzero bi-homogeneous components is at most
\begin{multline*}
  \sum_{j=l}^l \left(\frac{d_j^\bullet-r_j}{m}+1\right)  =
\left( \sum_{j=1}^l \frac{d_j^\bullet}{m}\right) - \left(\sum_{j=1}^l \frac{r_j}{m} \right) + l \\
=\frac{D_\vartheta}{m}-\frac{l-\rk\g_0}{2}+l=\bb(\g)-\bb(\g_0)+\rk\g_0=\bb(\g,\vth).
\end{multline*}
On the other hand, it follows from Eq.~\eqref{eq:alg-kr} that the total number of bi-homogeneous components of $H_1,\dots,H_l$ is at least $\bb(\g,\vartheta)$.
Therefore, all admissible bi-homogeneous components must be nonzero and algebraically independent.
\end{proof}

A precise relationship between $\gZ$ and $\gZ_\times$ depends on further properties of $\vth$. Two
complementary assertion are given below.

\begin{cl}     \label{cor1:main4}
In addition to the hypotheses of \emph{Theorem~\ref{free-main}},
suppose that  $\g_{(0)}$ has the {\sl codim}--$2$ property and $\g_0=\g^\vth$ is \emph{not} abelian.
Then $\gZ=\gZ_\times$ is the polynomial algebra freely generated by all nonzero bi-homogeneous components $H_{j,i}$.
\end{cl}
\begin{proof}
In this case, $0\in\BP_{\sf reg}$ and it follows from Theorem~\ref{thm:kot14}{\sf (iii)} that 
$\cz_0=\bbk[H_1^\bullet,\dots,H_l^\bullet]$, which is contained in $\gZ_\times$. On the other hand,
$\infty\in\BP_{\sf sing}$ (Corollary~\ref{cor:infty=reg}), hence $\cz_\infty$ is not required for $\gZ$.
Thus, $\gZ=\gZ_\times$.
\end{proof}

\begin{cl}      \label{cor2:main4}
In addition to the hypotheses of \emph{Theorem~\ref{free-main}},
suppose that $\g_{(0)}$ has the {\sl codim}--$2$ property, $\vth$ is inner, and $\g_0=\g^\vth$ \emph{is} abelian. Then $\cz_\infty=\gS(\g_0)$ and $\gZ=\mathsf{alg}\lg \gZ_{\times},\g_0\rg$ is a polynomial algebra.
 \end{cl}
\begin{proof}  
Recall that the subspace $\g_0$ lies in the centre of $\g_{(\infty)}$, hence $\gS(\g_0)\subset\cz_\infty$.
If $\vartheta$ is inner, then $\ind\g_{(\infty)}=\dim\g_0=\rk\g$ (Theorem~\ref{thm:ind-inf}). Hence
$\gS(\g_0)\subset \cz_\infty$ is an algebraic extension.
Since $\gS(\g_0)$ is an algebraically closed
subalgebra of $\gS(\g)$, it coincides with  $\cz_\infty$. As in the previous corollary, we have
$\cz_0\subset \gZ_\times$. Therefore, 
\[
  \gZ=\mathsf{alg}\lg \gZ_{\times}, \cz_\infty\rg=\mathsf{alg}\lg \gZ_{\times},\g_0\rg .
\]
Among the algebraically independent generators of $\gZ_\times$, one has $l$ nonzero functions
$H_{j,0}\in \gS(\g_0)$, $j=1,\dots,l$. (Note that $d_j^\bullet\in m\BZ$ for each $j$, since $\vth$ is inner.) Hence the 
passage from $\gZ_\times$ to $\gZ$ merely means that we have to replace $\{H_{j,0}\}$ with a basis
for $\g_0$.
\end{proof}

\begin{rmk}    \label{rem:Z-bad}
If $\vartheta$ does not admit a {\sf g.g.s.}, then $H_1^\bullet,\dots,H_l^\bullet$ are algebraically 
dependent. Since $\{H_i^\bullet\}$ are certain bi-homogeneous components, this means that
the number of all nonzero bi-homogeneous components of $\{H_j\}$ is larger than $\trdeg\gZ$.
Moreover, the case of involutions ($m=2$) shows that then the algebra 
$\cz_0=\gS(\g_{(0)})^{\g_{(0)}}$, which is contained in $\gZ$, is not polynomial, 
see~\cite[Sect.~5]{Y-imrn}. Therefore, it would be unwise to expect really good properties of $\gZ_\times$ 
or $\gZ$ without presence of g.g.s.
\end{rmk}

%%%%%%%%%%  Section %%%%%%%%%
\section{On the maximality problem} 
\label{sect:max}

Since $\gZ\subset \gS(\g)^{\g_0}$ by~\eqref{eq:Z-v-inv-g_0},  the algebra 
$\tilde\gZ:=\mathsf{alg}\langle\gZ,\gS(\g_0)^{\g_0}\rangle$ is still Poisson-commutative. 
It was proved in~\cite{OY} that if $\vth$ is an involution admitting a {\sf g.g.s.}, then
$\tilde\gZ$ is a {\bf maximal} Poisson-commutative subalgebra of $\gS(\g)^{\g_0}$. For $m\ge 3$, 
the same problem becomes more difficult, and we obtain only partial results in this section.
Our line of argument employs properties of graded polynomial algebras. 

Let $F_1,\ldots,F_N\in\mK[\{x_i\}]=\mK[\mathbb A^n]$ be algebraically independent 
homogeneous  polynomials.  Each differential $\textsl{d}F_i$ is a regular $1$-form on $\mathbb A^n$
with polynomial coefficients. Then 
\beq       \label{e-bff}
      \textsl{d}F_1\wedge\ldots \wedge \textsl{d} F_N=\bff R,
\eeq 
where $\bff\in\mK[\mathbb A^n]$ and $R$ is a regular differential $N$-form that is nonzero on a big 
open subset. Note that $\bff$ is defined uniquely up to multiplication with scalars. Let  
${\mathcal F}=\mK[\{F_j\}]$ be the subalgebra of $\mK[\mathbb A^n]$ generated by the polynomials 
$F_j$.

\begin{thm}[{\cite[Theorem~1.1]{ppy}}]       \label{ppy-max}  
If $\bff=1$, then ${\mathcal F}$ is an 
algebraically closed subalgebra of  $\mK[\mathbb A^n]$, i.e., if $H\in \mK[\mathbb A^n]$ is 
algebraic over the field\/ $\mK(F_1,\ldots,F_N)$, then
$H\in \mathcal F$. 
\end{thm}

Consider the following conditions on the Lie algebra $\g_{(0)}$:   
\begin{itemize}
\item[$(\lozenge_1)$]  \ $\ind\g_{(0)}=\rk\g=l$, i.e., $0\in\BP_{\sf reg}$; 
\item[$(\lozenge_2)$]  \ $\cz_0$ is a polynomial ring generated by $H_i^\bullet$ with $1\le i\le l$; 
\item[$(\lozenge_3)$]  \ $\dim\g_{(0),{\sf sing}}^*\le \dim\g-2$, i.e., $\g_{(0)}$ has the {\sl codim}--$2$
property;
\item[$(\lozenge_4)$] \ either $\g_0$ is non-abelian {\it or\/} $\g_0$ is abelian and $\vth$ is inner.
\end{itemize}
These conditions imply that $\gZ$ is a polynomial algebra, see Corollary~\ref{cor1:main4} and \ref{cor2:main4}. Moreover, the following is true.
\begin{prop}        \label{prop:tilde-Z-free}
If conditions $(\lozenge_1)$--$(\lozenge_4)$ are satisfied, then
$\tilde\gZ$ is a polynomial algebra, too.
\end{prop}
\begin{proof}
By Lemma~\ref{lm:outer-inv}, we have $H_{j,0}\ne 0$ if and only if $\vth(H_j)=H_j$, and the number 
of such $j$'s equals $\rk\g_0$. Then $H_{j,0}\in \gS(\g_0)^{\g_0}$.
In the passage from $\gZ$ to $\tilde\gZ$, these nonzero generators $H_{j,0}$ are replaced with the
basic symmetric invariants of $\g_0$. 
\end{proof}

\begin{lm}       \label{D}
Assume that $(\lozenge_1)$ and $(\lozenge_3)$ hold. 
Suppose that there is a divisor $D\subset\g^*$ such that $\dim\textsl{d}_\eta \gZ_\times < \bb(\g,\vth)$ for any 
$\eta\in D$. Then $D=Y\times \g_{>0}^*$, where $Y\subset\g_0^*\cap\g^*_{\sf sing}$
is a $G_0$-stable conical divisor in $\g_0^*$. 
\end{lm}
\begin{proof}
By Proposition~\ref{prop:codim-sing},  there is a big open subset $U_{\sf sr}\subset\g^*$ such that 
$U_{\sf sr}\subset \bigcap_{t\ne 0,\infty} \g_{(t),{\sf reg}}^*$.
By $(\lozenge_3)$, $\g_{(0),{\sf reg}}^*$ is also a big open subset of $\g^*$. 
Hence $\tilde U:=D\cap U_{\sf sr} \cap \g_{(0),{\sf reg}}^*$ is open and dense in $D$.
Take any $\eta\in \tilde U$.
If we write $\eta=\eta_0+\eta'$ with $\eta\in\g_0^*$ and $\eta'\in\g_{>0}^*$, 
then we may also assume that 
$\eta_0\in (\g_0)^*_{\sf reg}$. Assume that $\eta_0\in\g^*_{\sf reg}$.
Then $\eta\in\g^*_{\infty,{\sf reg}}$ by Lemma~\ref{lm:xi}.  
Thus, in that case, $\eta\in U_{\sf srr}$. Moreover, such $\eta$ is generic in the sense of Lemma~\ref{rank-ker} and the conclusion of that lemma holds for it.  
Arguing as in the proof of Theorem~\ref{thm-trdeg} and using \eqref{sum-dif},
we obtain that $\dim L(\eta)=\bb(\g,\vth)$, a contradiction! 

Therefore we must have $\eta_0\in \g^*_{\sf sing}$ for a generic $\eta\in D$ and hence for any $\eta\in D$. 
Since $\g_0^*\cap\g^*_{\sf sing}$ is a proper subset of $\g_0^*$, 
the divisor $D$ is indeed of the form $Y\times \g_{>0}^*$, where $Y\subset \g_0^*\cap\g^*_{\sf sing}$. 

The algebra $\gZ_\times$ consists of $G_0$-invariants and the group $G_0$ is connected. 
Thereby each irreducible component of the subset 
$\{\gamma\in\g^* \mid \dim\textsl{d}_\gamma \gZ_\times<\bb(\g,\vth)\}$ is $G_0$-stable. In particular, 
$D$, and hence $Y$ as well, is $G_0$-stable. Since $\gZ_\times$ is a homogeneous subalgebra,
the divisor $Y\subset\g_0^*$ is conical.  
\end{proof}

\begin{thm}            \label{thm:no-D}
Suppose that $(\lozenge_1)$,  $(\lozenge_2)$, and $(\lozenge_3)$ hold. 
If\/ $\g_0\cap\g^*_{\sf sing}$ does not contain divisors, then $\gS(\g_0)^{\g_0}\subset\gZ_\times$ and 
$\gZ_\times=\gZ=\tilde\gZ$ is a maximal Poisson-commutative subalgebra of\/ $\gS(\g)^{\g_0}$.  
\end{thm}
\begin{proof}
We know that $\gZ_\times$ is a polynomial algebra (Theorem~\ref{free-main}) and 
$\trdeg\gZ_\times=\bb(\g,\vartheta)$ (Theorem~\ref{thm-trdeg}). 
Suppose that $\gZ_\times\subset \ca\subset  \gS(\g)^{\g_0}$ and $\ca$ is Poisson-commutative. Then we have 
$\trdeg \ca\le \bb(\g,\vartheta)$ by~\eqref{eq:MY19}.  In view of Lemma~\ref{D}, 
the differentials $\textsl{d} F_i$ of the algebraically independent generators $F_i\in\gZ_\times$ 
are linearly independent on a big open subset. Then by Theorem~\ref{ppy-max}, 
$\gZ_\times$ is an algebraically closed subalgebra of $\gS(\g)$. In particular, we must have 
$\ca=\gZ_\times$. This applies to $\gZ$ and $\tilde\gZ$ as well. 
\end{proof}

\begin{rmk}
Conditions $(\lozenge_1)$--$(\lozenge_4)$ are satisfied for involutions $\vartheta$ that have {\sf g.g.s.}.  
\end{rmk}

\begin{ex} \label{ex-max}
Theorem~\ref{thm:no-D} applies to several outer automorphisms of semisimple Lie algebras, 
for instance, to $\vth\in\Aut(\g)$ of order $2m$ in case $\g=(\gt{sl}_{2n})^{m}$ and
$\g_0$ is a diagonally embedded $\gt{sp}_{2n}$, cf.~\cite[(4.2)]{p09}. 
Automorphisms of this form are considered in Section~\ref{sec-twist}.  
\end{ex}

Theorem~\ref{thm:from-09}  provides a bunch of automorphisms
$\vth$ with $\oth\ge 3$ such that $\tilde\gZ$ is a polynomial algebra,  
see Examples~5.9 and 5.10 in~\cite{p09}. In all these cases, $\gZ\ne\tilde\gZ$ and   
we conjecture that $\tilde\gZ$ is a maximal Poisson-commutative subalgebra of 
$\gS(\g)^{\g_0}$.  

Summarising our previous considerations, we can say that in order to guarantee some good properties of 
the Poisson-commutative subalgebras $\gZ_\times$,$\gZ$, and $\tilde\gZ$,  the following properties of
$\vth\in\Aut(\g)$ and thereby of $\g_{(0)}$ are needed:
\begin{itemize}
\item[\sf (a)] \ $\ind\g_{(0)}=\ind\g$;
\item[\sf (b)] \ $\g_{(0)}$ has the {\sl codim}--$2$ property;
\item[\sf (c)] \ $\vth$ admits a {\sf g.g.s.}
\end{itemize}
If $\oth=2$, then {\sf (a)} and {\sf (b)} are always satisfied~\cite{p07}, and a complete description of involutions
admitting a {\sf g.g.s.} is available~\cite{Y-imrn}.
In a forthcoming paper, we are going to undertake a thorough substantial investigation of these properties
for arbitrary $\vth$.

%%%%%%%%%%%     Section   %%%%%%%%%%%%%%%%%
\section{The case of a cyclic permutation} 
\label{sect:cyc} 

\noindent
Let $\h$ be a simple non-abelian Lie algebra and $\g=\h^{m}$ the direct sum of $m\ge 2$ copies of 
$\h$. Then $l=\rk\g=m{\cdot}\rk\h$.
Let $\vth\in\Aut(\g)$ be a cyclic permutation of the summands of $\g$. That is,
\[
    \vth(x_1,\dots,x_m)=(x_m,x_1,\dots,x_{m-1}), \ x_i\in\h .
\]
Then, for $i=0,1,\dots,m-1$ and $\zeta=\sqrt[m]1$, we have 
\beq    \label{eq:isom}
     \g_i= \{(x,\zeta^{-i}x,\dots,\zeta^{(1-m)i}x)\mid x\in\h\} .
\eeq
In particular, $\g_0=\Delta_\h\simeq\h$ is the diagonal and each $\g_i$ is isomorphic to $\h$ as vector 
space and as $\g_0$-module.
Here the Lie algebra $\g_{(0)}$ is isomorphic to the truncated current algebra
\[
   \hm:=\h\otimes \bbk[t]/(t^{m})=\h[t]/(t^{m}),
\]
see~\cite[Corollary\,3.6]{p09}. The isomorphism $\g_{(0)}\simeq \hm$ takes $\g_i$ to
the image of $\h t^i\subset \h\otimes \bbk[t]$ in $\hm$.
The Lie algebra $\hm$ is also known as a (generalised) {\it Takiff algebra} modelled on $\h$.
By~\cite[Theorem\,2.8]{rt}, we have $\ind\g_{(0)}=m{\cdot}\rk\h =\rk\g$, i.e., $0\in\BP_{\sf reg}$. 
It then follows from Theorem~\ref{thm-trdeg} that in this case 
\[
\trdeg\gZ=\bb(\h^{m},\vth)=\frac{1}{2}\bigl((m-1)\dim\h+(m+1)\rk\h\bigr)=(m-1)\bb(\h)+\rk\h .
\]
On the other hand, $\g_0$ is not abelian, hence $\BP_{\sf sing}=\{\infty\}$ and $\cz_\infty$ is not 
required for $\gZ$. Since $\hm$ is $\BN$-graded and the zero part is semisimple, the nilpotent radical 
$\hm^{u}$ is equal to $\bigoplus_{i=1}^{m-1} \h t^i$. Comparing this with Proposition~\ref{prop:graded-limits} 
on the graded structure of $\q_{(\infty)}$, we conclude that here $\g_{(\infty)}\simeq\hmi^{u}$. 
Since $\ind\g_{(\infty)}$ is computed for any $\vth$ in Theorem~\ref{thm:ind-inf}, we obtain a new 
observation that
\beq      \label{eq:ind-nil-T}
   \ind\hm^{u}=(m-2)\rk\h+\dim\h .
\eeq
Upon the identification $\g_1\simeq\h$, see Eq.~\eqref{eq:isom}, an element $x\in\h$ is nilpotent 
(resp. semisimple, regular) in $\h$ if and only if $(x,\zeta^{-i}x,\dots,\zeta^{(1-m)i}x) \in\g_1$ is 
nilpotent (resp. semisimple, regular) in $\g$. This also implies that the null-cone $\cN_1$ is isomorphic to 
the null-cone of $\h$. Hence $\cN_1$ is irreducible. 
Thus, $\vth$ is both $\gS$-regular and very $\gN$-regular. 

Let $\{H_1,\dots,H_l\}$ be s set of homogeneous generators of $\gS(\g)^\g$ consisting of 
$\vth$-eigenvectors. Since $\vth$ is $\gS$-regular and very $\gN$-regular, it follows from 
Theorem~\ref{thm:from-09} that $\g_{(0)}$ has the {\sl codim}--$2$-property and
$\cz_0=\bbk[H^\bullet_1,\dots,H^\bullet_l]$, see also \cite{rt}. The last relation also means that $\{H_j\}$ 
is a {\sf g.g.s.} for $\vth$.

\begin{thm}     \label{thm-cyc} 
If\/ $\g=\h^m$ and $\vth$ is a cyclic permutation, then the algebra $\gZ=\gZ(\h^m,\vth) $ is freely 
generated by the nonzero bi-homogeneous components $H_{j,i}$, $1\le j\le l=m{\cdot}\rk\h$. Moreover, 
$\gZ$ is a \emph{maximal} Poisson-commutative subalgebra of\/ $\gS(\g)^{\g_0}$.
\end{thm}
\begin{proof}
The above discussion shows that conditions $(\lozenge_1)$--$(\lozenge_4)$ hold for $\vth$.  
Hence $\gZ$ is freely generated by the nonzero bi-homogeneous components $H_{j,i}$ by Corollary~\ref{cor1:main4}.

A point $\xi\in\g_0^*$ is regular in $\g^*$ if and only if $\xi$ is regular in $\g_0^*\simeq\h^*$. 
Thereby  
\[   
      \dim(\g_{0}^*\cap \g^*_{\sf sing})=\dim\h-3  
\] 
and this intersection does not contain divisors of $\g_0^*$. Therefore $\gZ$   
is a maximal  Poisson-com\-mu\-ta\-tive  subalgebra of $\gS(\g)^{\g_0}$ by Theorem~\ref{thm:no-D}. 
\end{proof}

It is not hard to produce a generating set $\{H_j\}\subset\gS(\g)^{\g}$ consisting of
$\vth$-eigenvectors. Suppose that $\rk\h=r$ and that $\{ F_1,\dots, F_r \}$ is a generating set of 
$\gS(\h)^{\h}\subset \gS(\g)$
for the first copy of $\h$ in $\g$. For $0 \le k < m$, set 
\begin{equation}              \label{ik}
       F_i^{[k]}=\frac{1}{m}\left(F_i+\zeta^k\vth(F_i)+\zeta^{2k}\vth^2(F_i)+\ldots + \zeta^{k(m-1)}\vth^{m-1}(F_i)\right).
\end{equation}  
Then $\vth(F_i^{[k]})=\zeta^{-k} F_{i}^{[k]}$. Thus, we can take
$\{H_j \mid 1\le j\le l\}=\{F_{i}^{[k]}\mid 1\le i\le r, 0\le k < m\}$. 

In Section~\ref{sect:Gaudin}, we will need an explicit description of the bi-homogeneous 
components $H_{j,i}$, which exploits the polarisation construction of \cite{rt}. Suppose that each $F_i$ is 
homogeneous and  $b_i=\deg F_i$. 
 
For $\xi\in\h^*$, let $\xi^{(k)}\in\g^*$ be the corresponding element sitting in the $k$-th copy of $\h^*$ in 
$\g^*$.   Then the induced action of $\vth$ in $\g^*$ is given by the formula:
\[
    \vth\bigl(\xi^{(1)}_1+\ldots+\xi^{(m-1)}_{m-1}+ \xi^{(m)}_m\bigr)=\xi^{(1)}_2+\ldots+\xi^{(m-1)}_m+ \xi^{(m)}_1 .
\]
Accordingly, $\g_j^*=\{\xi^{(1)}+\zeta^j\xi^{(2)}+\ldots+\zeta^{(m-1)j} \xi^{(m)}\mid \xi\in\h^*\}$ and
we fix the corresponding isomorphisms $\nu_j: \h^*\isom \g_j^*$, where 
\beq         \label{iso-s}
          \xi \overset{\nu_j}{\longmapsto} \left(\xi^{(1)}+\zeta^j\xi^{(2)}+\ldots+\zeta^{(m-1)j} \xi^{(m)}\right) \in \g_j^*.
\eeq
For $\xi_0,\ldots,\xi_{m-1}$ with $\xi_j\in \g_j^*$, let 
$\bar\xi=(\xi_0,\ldots,\xi_{m-1})$ be the corresponding element of $\g^*$ and 
$\xi_0+\ldots+\xi_{m-1}$ be an element of $\h ^*$. 
For  $s\in\mK$, set $\phi_s(\bar\xi)=(s^{m-1}\xi_0,s^{m-2}\xi_1,{\ldots},\xi_{m-1})$ and 
define the {\it $\phi$-polarisations} $F_{[k]}$ of $F\in\gS^d(\h )$  in the following way 
\begin{multline}         \label{s-pol}
      F(\phi_s(\bar\xi)):=F(s^{m-1}\xi_0{+}s^{m-2}\xi_1{+}\ldots{+}\xi_{m-1})=: \\
      F_{[0]}(\bar\xi)+sF_{[1]}(\bar\xi) +\ldots +s^{d(m-1)}F_{[d(m-1)]}(\bar\xi).
\end{multline}
It is convenient to assume that $F_{[k]}=0$ for $k>d(m-1)$. 

\begin{ex}           \label{ex-pol-phi}
If $x\in\gt h$ and $0\le k<m$, then $x_{[k]}=
\frac{1}{m}(x,\zeta^{-i}x,\dots,\zeta^{(1-m)i}x)\in\g_{i}$ with $i=m-1-k$.
\end{ex}

According to~\cite{rt}, $\cz_0$ is freely generated by $(F_i)_{[k]}$ with $1\le i\le r$, $0\le k<m$, where 
$\gS(\h)^\h=\mK[F_1,\ldots,F_r]$.  

\begin{prop}                 \label{pol}
We have  $\{H_{j,i}\,{\ne}0 \mid 1{\le} j {\le} l, \,0{\le} i\}=\{(F_i)_{[k]} \mid 1\le i\le r, 0\le  k\le b_i(m{-}1)\}$, 
where $(F_i)_{[k]}$ are the $\phi$-polarisations of $F_i$  defined by~\eqref{s-pol}. 
\end{prop}
\begin{proof}
Now we need to distinguish the first copy $\gt h^{(1)}\subset\gt g$ of $\h$ from an abstract $\h$. 
Suppose that 
$F^{(1)}\in\gS^d(\gt h^{(1)})$ is the image of $F\in\gS^d(\gt h)$ under the tautological isomorphism 
$\gS(\gt h)\simeq\gS(\gt h^{(1)})$. 
The combination of~\eqref{iso-s} and the definition of $\varphi_s$, see Section~\ref{subs:contr-&-inv}, 
leads to 
\[
     \varphi_s(F^{(1)})=F_{[(m-1)d]} + sF_{[(m-1)d-1]}+\ldots 
+s^{(m-1)d-1} F_{[1]}+s^{(m-1)d} F_{[0]}
\]
for $s\in\mK^\star$. In this notation, $\vth(F_{[k]})=\zeta^{-d-k}F_{[k]}$. Next we plug the formula for 
$\varphi_s(F^{(1)})$ into~\eqref{ik} and conclude that each bi-homogeneous component of each 
$F_{i}^{[k]}$ is a $\phi$-polarisation of $F_{i}\in\gS^{b_i}(\gt h)$.  By Theorem~\ref{thm-cyc},
the total number of the nonzero bi-homogeneous components $(F_{i}^{[k]})_{i'}$ is 
$\trdeg\gZ=\bb(\g,\vth)=(m-1){\cdot}\bb(\h)+r$. Since $\h$ is semisimple, we have  
$\bb(\h)=\sum_{i=1}^r b_i$. Finally, the total number of $\phi$-polarisations $(F_i)_{[k]}$ equals 
\[ 
   r+(m-1)\sum_{i=1}^r b_i=r+(m-1){\cdot}\bb(\h)=\bb(\g,\vth) . 
\]
Hence the two sets in question coincide. 
\end{proof}

%%%%%%%%%%%    Section   %%%%%%%%%%%%%%%
\section{Gaudin subalgebras}
\label{sect:Gaudin}

\noindent
Let $\h $ be the same as in the previous section. The enveloping algebra $\U(t^{-1}\h [t^{-1}])$
contains a remarkable  commutative subalgebra $\z(\widehat{\h })$, which is known as
the {\it Feigin--Frenkel centre} ~\cite{ff:ak} (see also \cite{f:lc}). This subalgebra is defined as the centre 
of the universal affine vertex algebra associated with the affine Kac--Moody algebra $\widehat{\h }$
at the critical level. In particular, each element of $\z(\widehat{\h })$ is annihilated by the
adjoint action of $\h $. The elements of $\z(\wg)$ give rise to higher Hamiltonians of the Gaudin model, 
which describes a completely integrable quantum spin chain \cite{FFRe}.

A Gaudin model consists of $n$ copies of $\h $ and the Hamiltonians
\[
{\mathcal H}_k=\sum_{j\ne k} \frac{\sum_{i=1}^{\dim\h } x_i^{(k)} x_i^{(j)}}{z_k-z_j}  , \enskip 1\le k\le n, 
\]
where $z_1,\ldots,z_n\in\mK$ are pairwise distinct and $\{x_i^{(k)}\mid 1\le i\le\dim\h \}$ is 
a basis for the $k$-th copy of $\h$ that is orthonormal w.r.t. $\varkappa$. 
Letting $\g=\h^{n}$, these Gaudin Hamiltonians $\{\mathcal H_k\}$ can be regarded as elements of either 
$\U(\h )^{\otimes n}\simeq\U(\g)$ or $\gS(\g)$.  
They commute (and hence Poisson-commute) with each other. Then higher Gaudin Hamiltonians are 
 elements of $\U(\h )^{\otimes n}$ that commute with all ${\mathcal H}_k$.

Recall the construction of~\cite{FFRe} that provides a {\it Gaudin subalgebra} ${\mathcal G}$
of $\U(\h )^{\otimes n}$. 
Set $\wg_-=t^{-1}\h [t^{-1}]$ and let $\Delta\U(\wg_-)\simeq\U(\wg_-)$ be the diagonal of 
$\U(\wg_-)^{\otimes n}$. Then any vector $\vec z=(z_1,\ldots,z_n)\in (\bbk^\star)^n$ defines 
a natural homomorphism $\rho_{\vec z}\!:\Delta\U(\wg_-) \to \U(\h )^{\otimes n}$, where 
\[
 \rho_{\vec z}(x t^k) = z_1^k x^{(1)}+z_2^k x^{(2)}+\ldots + z_n^k x^{(n)} \ \text{ for } \ x\in\h.
\] 
By definition, ${\mathcal G}={\mathcal G}(\vec z)$ is the image of $\z(\wg)$ under $\rho_{\vec z}$, and 
one can prove that ${\mathcal G}$ contains ${\mathcal H}_k$ for all $k$. 
Hence $[{\mathcal G}(\vec z),{\mathcal H}_k]=0$ for each $k$.   
One also has ${\mathcal G}\subset \U(\g)^{\Delta\h}$ by the construction. 

\begin{rmk}
Gaudin subalgebras have attracted a great deal of attention, see e.g. \cite{G-07} and references therein. 
It is standard to work with complex Lie algebras in this framework. Gaudin algebras are closely related 
to {\it quantum MF  subalgebras} \cite{r:si} and share some of their properties. In particular, for a generic 
$\vec z\in(\mathbb R)^n$, the action of $\Gz$ on an irreducible finite-dimensional 
$\g$-module $V(\lambda_1)\otimes\ldots\otimes V(\lambda_n)$ is diagonalisable and has a simple 
spectrum on the subspace of highest weight vectors w.r.t.~the diagonal 
$\h\simeq\Delta_\h\subset \h^{\oplus n}$~\cite{L-Gau}. 
\end{rmk}

The cyclic permutation $\vth$ is an automorphism of $\g$ of order $n$. 
Let now $\zeta$ be a primitive  $n$-th  root  of unity. 
Then the $\vth$-eigenspace $\g_j$ corresponding to $\zeta^j$ is 
\beq       \label{gj-n}
\g_j=\{x^{(1)} + \zeta^{-j} x^{(2)}+\zeta^{-2j} x^{(3)}+ \ldots +\zeta^{j(1-n)} x^{(n)} \mid  x \in \h\}\simeq \h.
\eeq  
Let $\bar t$\/ denote the image of $t^{1-n}$ in $\mK[t^{-1}]/(t^{-n}-1)$. 
Then the quotient $\h[t^{-1}]/(t^{-n}-1)$ has the canonical $\BZ_n$-grading
$\h\oplus \h \bar t\oplus \h\bar t^2 \oplus\ldots\oplus \h{\bar t}^{n-1}$. In particular, 
$\h[t^{-1}]/(t^{-n}-1)\simeq \g$ as a $\Z_n$-graded Lie algebra. Fix 
an isomorphism $\g_j\to \bar t^j \h$ as the projection {\sf pr}$_1\!: \g\to\h$ on the first summand 
of $\g$ combined with the multiplication by $\bar t^j$. 
Now we  regard $\g$ as the quotient of $\h[t^{-1}]$.   

The above identifications related to $\vth$ provide a simpler approach to constructing certain
Gaudin subalgebras.

\begin{prop}    \label{prop:new-G}
Take $z_k=\zeta^{1-k}$ and consider the corresponding Gaudin subalgebra $\cG=\Gz$ in 
$\U(\h)^{\otimes n}$. Then
$\cG$ coincides with the image of $\z(\wh)$ in the quotient $\U(\wh_-)/(t^{-n}{-}1)=\U(\g)$.
\end{prop}
\begin{proof}
Take  $a\in\mathbb Z_{>0}$ and write $a=nq+j$ with $1\le j\le n$. Then
for $x t^{-a} \in \h[t^{-1}]$, we have
\[
 \rho_{\vec z}(x t^{-a})=x^{(1)} + \zeta^{j} x^{(2)} +\zeta^{2j} x^{(3)}+ \ldots + \zeta^{j(n-1)}x^{(n)}\in\g_{n-j}.
\]
Therefore $\rho_{\vec z}(x t^{-a})$ identifies with the image $x\bar t^{n-j}$ of $x t^{-a}$ in 
$\h[t^{-1}]/(t^{-n}{-}1)$.
\end{proof}

By a theorem of Feigin and Frenkel~\cite{ff:ak}, $\z(\widehat{\h})$ is a polynomial ring in 
infinitely many variables, with a distinguished set of generators.  Set $\tau=-\partial_t$.  There are 
algebraically independent elements $S_1,\ldots,S_{r} \in \z(\wg)$   such that 
\[
    \z(\widehat{\h})= \bbk[\tau^k (S_i) \mid 1\le i\le r, k\ge 0].
\]
The symbols  $\gr\!(S_i)$ are homogeneous elements of $\gS(\h  t^{-1})$ and if 
$F_i=\gr\!(S_i)|_{t=1}$, then $\bbk[F_1,\ldots,F_r]=\gS(\h)^{\h}$. 
The set $\{S_i\}$ is called a {\it complete set of Segal--Sugawara vectors}.  
Keep the notation $b_i=\deg F_i$. 

\begin{lm} \label{lm-pol}
For $k\ge 0$, let $\bF_{i,k}$ be the symbol of $\rho_{\vec z}(\tau^k(S_i))$. 
In the case  $z_k=\zeta^{1-k}$ for $k=1,\dots,n$, we have then  
\[
    \langle \bF_{i,k} \mid 0\le k\le (n{-}1) b_i\rangle_{\mK}=
\left<(F_i)_{[k]}\mid0\le k\le (n{-}1) b_i\right>_{\mK}
\] 
for the $\phi$-polarisations $(F_i)_{[k]}$ related to $\g^*=\bigoplus_{j=0}^{n{-}1} \g_j^*$, cf.~\eqref{s-pol}. 
\end{lm}
\begin{proof}
Since $\gr\!(S_i)\in\gS(t^{-1}\h)$, we have
$\bF_{i,0}\in\gS(\g_{n-1})$ and  $\bF_{i,0}=n^{b_i}(F_i)_{[0]}$. 
Assuming $n\ge 2$, we can state that  $\bF_{i,1}=n^{b_i}(F_i)_{[1]}$. 
If $k<n$, then clearly $\frac{1}{k!}\bF_{i,k} =n^{b_i}  (F_i)_{[k]}$.   
More generally, as long as $k\le (n{-}1)b_i$, we have  
\[
    \frac{1}{k! n^{b_i} }\bF_{i,k} =  (F_i)_{[k]} + \sum_{0<j\le k/n} \binom{b_i-j-1}{j} (F_i)_{[k-jn]}, 
\]
where the leading term of $\frac{1}{k!n^{b_i} }\bF_{i,k}$ corresponds to  those summands of
$\gr\!(\tau^k(S_i))$ that belong to $\gS(\h t^{-n}{\oplus}{\ldots}{\oplus}\h t^{-1})$. 
\end{proof}

Let $\gZ\subset\gS(\g)$ be the  Poisson-commutative subalgebra associated with $\vth$. 

\begin{thm}            \label{q}
If $z_k=\zeta^{1-k}$ for $k=1,\dots,n$, then the Gaudin algebra $\Gz\subset\U(\g)$ is a 
quantisation of\/ $\gZ$.
\end{thm}
\begin{proof}
By Theorem~\ref{thm-cyc}, $\gZ$ is generated by the bi-homogeneous components 
$H_{j,i}$ with $1\le j\le l$. Proposition~\ref{pol} provides a  description of these components.  
Then by Lemma~\ref{lm-pol},  we have $\gZ\subset\gr\!(\Gz)$. 
Clearly, $\gr\!(\Gz)\subset\gS(\g)^{\g_0}$ is a Poisson-commutative 
subalgebra.  Recall that $\gZ$ is a maximal Poisson-commutative subalgebra of $\gS(\g)^{\g_0}$ by
Theorem~\ref{thm-cyc}. Thus $\gZ =\gr\!(\Gz)$ as required. 
\end{proof}

\begin{rmk}
If $z_k\ne z_j$ for $k\ne j$, then $\trdeg\Gz=(n-1)\bb(\h)+\rk\h$ by~\cite[Prop.~1]{G-07}.
That is, $\trdeg\Gz=\bb(\g,\vth)$ is maximal possible for commutative subalgebras of $\U(\g)^{\g_0}$.
Theorem~\ref{q} shows that one particular Gaudin algebra is actually a maximal commutative 
subalgebra of $\U(\g)^{\g_0}$. Using a connection with the MF subalgebras established in~\cite{r:si},
we can also prove that any such $\Gz$ is a maximal commutative subalgebra of $\U(\g)^{\g_0}$. 
\end{rmk}

\section{Fixed-point subalgebras in the infinite dimensional case} 
\label{sec-twist}

Let $\h $ be the same as before.  Let now $\vartheta\in\Aut(\h)$ be an 
automorphism of order $m$ and  $\zeta=\sqrt[m]{1}$ be primitive.  To any such $\vth$, one associates a
{\it $\vth$-twisted loop algebra} $\h [t,t^{-1}]^{\vartheta}=(\h [t,t^{-1}])^{\vartheta}$, where 
$\vartheta(t)=\zeta^{-1} t$ and $\vth(t^{-1})=\zeta t^{-1}$. 
If $\vth$ is an outer automorphism of $\gt h$, then  $\h [t,t^{-1}]^{\vartheta}$
is a subalgebra of a twisted Kac--Moody algebra, see \cite[Chap.~8]{kac} for details. Similarly to the case  of $\wg_-$, we identify $\wg_-^\vth:=(\wg_-)^\vth$ with 
$(\h [t,t^{-1}]/\h [t])^\vth$.  This leads to an action of $(\h [t])^\vth$ on 
$\wg_-^\vth$ and correspondingly on $\gS(\wg_-^\vth)$.
In this section, we consider the  invariants of $(\h [t])^\vth$ in $\gS(\wg_-^\vth)$, i.e., 
the subalgebra
\[
\gZ(\wh_-,\vth)=\gS(\wg_-^\vth)^{(\h [t])^\vth} , 
\] 
which can be thought of as a $\vth$-twisted analogue of $\gr\!(\z(\wh))$, cf.~\eqref{hat-inv}. 

Set $\g=\h^{n}$. Let  $\tilde\vartheta\in\Aut(\g)$ be the composition of $\vartheta$ applied to one copy 
of $\h$  only  and a cyclic permutation of the summands. Formally speaking, 
\[\tilde\vth(y_1^{(1)}+y_2^{(2)}+\ldots +y_n^{(n)})=y_n^{(1)}+\vth(y_1)^{(2)}+y_2^{(3)}+\ldots + y_{n-1}^{(n)} 
\]
for any $y_1,\ldots,y_n\in\h$.  
The order of $\tilde\vartheta$ is $N=nm$. Similarly to the case of $\wg_-$, there are isomorphisms   
$\g\simeq \h [t]^\vartheta/(t^{N}{-}1)$ and $\g_{(0)}\simeq \h [t]^\vartheta/(t^{N})$. 

Set $\q=\h _{(0)}$ transferring the notation of Section~\ref{sect:3}  to $\h $. 
Next we want to understand the connection between the Takiff algebra $\qqn=\q[t]/(t^n)$, modelled on $\q$,  and $\g_{(0)}$. Let $\tilde\zeta=\sqrt[N]{1}$ be such that $\zeta=\tilde\zeta^n$. 
 Set also $\omega=\tilde\zeta^m$. 

\begin{lm}[{cf. \cite[Sect.~3]{p09}}]   
\label{twist}
{\sf (i)} The Takiff algebra  $\qqn$ is a contraction of $\g_{(0)}$. 
\\[.2ex]
{\sf (ii)} If\/ $\ind\q=\rk\h $, then $\ind\g_{(0)}=\rk\g$. 
\\[.2ex]
{\sf (iii)} If there is a {\sf g.g.s.} for $\vartheta$, then there is a {\sf g.g.s.} for $\tilde\vartheta$. 
\\[.2ex]
{\sf (iv)} If\/ $\q$ has the  {\sl codim}--$2$ property, then so does $\g_{(0)}$.
\end{lm}
\begin{proof}
{\sf (i)} \ Let $\g=\bigoplus_{0\le j<N} \g_j$ be the $\BZ_N$-grading of $\g$ defined by 
$\tilde\vartheta$. Recall that $\g_{(0)}$ is $\BN$-graded and the Lie bracket in $\g_{(0)}$ is defined by
$[\g_j,\g_{j'}]_0=[\g_j,\g_{j'}]\subset\g_{j+j'}$ if $j+j'<N$ and $[\g_j,\g_{j'}]_0=0$ otherwise. 
The Lie algebra $\qqn$ is also $\mathbb N$-graded and has the same components, however, the Lie bracket is 
different. Write $j=qm+i$ with $0\le i<m$ and similarly $j'=q'm+i'$. Then in 
$\qqn$, the bracket between $\g_j$ and $\g_{j'}$ is the same as in $\g_{(0)}$  
if $i+i'<m$. If $i+i'\ge m$, then $\g_j$ and $\g_{j'}$ commute in $\qqn$. 
This shows that $\qqn$ is a contraction of $\g_{(0)}$ corresponding to the following linear 
map:
\[
\tilde\psi_s\!: \g_{(0)}\to \g_{(0)} \ \text{ with } \ s\in\mK^\star, \ {\tilde\psi_s}|_{\g_j}= s^i \id_{\g_j}.
\]

{\sf (ii)}  By \cite[Th\'eor\`eme~2.8{\sf (i)}]{rt}, one has $\ind\qqn=n{\cdot}\ind\q$.
Therefore, if $\ind\q=\rk\h $, then $\ind\qqn=n{\cdot}\ind\q=\rk\g$. Because index cannot decrease under contractions, 
$\rk\g=\ind\g\le \ind\g_{(0)}\le \ind\qqn=\rk\g$.  This settles {\sf (ii)}. 
\\ \indent
{\sf (iii)} Let $F_1,\ldots,F_r\in\gS(\h)^\h$ be a {\sf g.g.s.} for $\vartheta$ such that 
$\vth(F_i)=\zeta^{\ell_i} F_i$ with $\ell_i\in\Z$ for each $i$. 
Regarding $F_i$ as an element of $\gS^{b_i}(\gt h^{(1)})$, set 
\begin{equation}              \label{ik-tilde}
       F_i^{[k]}=\frac{1}{n}\left(F_i+\omega^k\tilde\zeta^{-\ell_i}\tilde\vth(F_i)+\omega^{2k}\tilde\zeta^{-2\ell_i}\tilde\vth^2(F_i)+\ldots + \omega^{(n-1)k}\tilde\zeta^{(1-n)\ell_i}\tilde\vth^{n-1}(F_i)\right),
\end{equation}  
cf.~\eqref{ik}.  Here $\tilde\vth^n(F_i)=\vth(F_i) = \zeta^{\ell_i}F_i\in\gS(\gt h^{(1)})$. Furthermore,  
\[
\tilde\vth(F_i^{[k]})=\omega^{(n-1)k} \tilde\zeta^{\ell_i} F_i^{[k]}, \ \ \text{ where } \ \ 
\omega^{(n-1)k} \tilde\zeta^{\ell_i} =\tilde\zeta^{\ell_i-mk}.
\] 
We claim that 
\[
     \{H_{j}\mid 1\le j \le l\}=\{F_{i}^{[k]} \mid 1\le i\le r, 0\le k < n\}
\] 
is a {\sf g.g.s.} for $\tilde\vartheta$.

Let $\tilde\varphi\!: \bbk^\star\to {\rm GL}(\g)$ be the polynomial representation such that 
$\tilde\vp_s(x)=s^j x$ for $x\in \g_j$. Set $B_i^\bullet = b_i^\bullet+b_i(n-1)m$. Here 
$B_i^\bullet \equiv b_i^\bullet  \equiv \ell_i \pmod m$. Furthermore, 
$\deg_{\tilde\vp} F_{i}^{[k]} \le B_i^\bullet$, because $\tilde\vth^{\upsilon}(\gt h_j^{(1)})\subset \bigoplus_{q=0}^{n{-}1} \gt g_{qm+j}$ for 
$0\le j<m$ and $0\le\upsilon$. However, the equality $\deg_{\tilde\vp} F_{i}^{[k]} = B_i^\bullet$ is 
possible only if $B_i^\bullet -\ell_i+mk\in N\Z$. If this is not the case for $k$, then 
$\deg_{\tilde\vp} F_{i}^{[k]} \le B_i^\bullet-m$. And again the equality is possible only if 
$B_i^\bullet  -\ell_i+m(k-1)\in N\Z$. Iterating these arguments, we obtain that 
$\sum_{k=0}^{n-1} \deg_{\tilde\vp} F_{i}^{[k]} \le n b_i^\bullet  + m(2b_i-1) \frac{n(n-1)}{2}$ 
for each $1\le i\le r$. Hence 
\[
   \sum_{1\le j\le l}\deg_{\tilde\vp } H_{j}\le n\left( \sum_{i=1}^r b_i^\bullet \right)+m\frac{n(n-1)}{2}\sum_{i=1}^r(2b_i-1)= n D_\vartheta+ N\frac{n-1}{2}\dim\h =D_{\tilde\vth}.
\]
Now the claim follows from Theorem~\ref{thm:kot14}.

{\sf (iv)} \ Suppose that  $\q$ has the {\sl codim}--$2$ property, then so does the Takiff algebra 
$\qqn$~\cite{rt,kot-T}.  Furthermore, if the index does not change under a polynomial contraction, then 
the dimension of the singular subset can only increase~\cite[(4.1)]{Y-imrn}. Applying 
this to $\g_{(0)}$ and its contraction $\qqn$, we conclude that  $\g_{(0)}$ has the {\sl codim}--$2$ 
property.   
\end{proof}

\begin{thm}           \label{twist-t}
{\sf (i)} If\/ $\ind\h _{(0)}=\rk\h $, then $\gZ(\wh_-,\vth)$ is a 
Poisson-commutative subalgebra. \\[.2ex]
{\sf (ii)} The algebra $\gZ(\wh_-,\vth)$  
is a polynomial ring (in infinitely many variables) if and only if\/ $\gS(\g_{(0)})^{\g_{(0)}}$ is a 
polynomial ring for each $n\ge 1$ and $\g=\h^{n}$ with $\tilde\vth\in\Aut(\g)$ as above.
\end{thm}
\begin{proof}
For $N=nm$, set  
\[
     \mathbb W_N:=(\h  t^{-N}\oplus\ldots\oplus\h  t^{-1})^\vartheta\subset\wg_-^\vth.
\] 
Then  $\h [t]^{\vartheta}$ acts on $\mathbb W_N$ via its quotient $\g_{(0)}$ and 
$\mathbb W_N\simeq \g_{(0)}$ as a $\g_{(0)}$-module.
Therefore there is an isomorphism of commutative algebras 
$\gS(\mathbb W_N)^{\h [t]^{\vartheta}}\simeq \gS(\g_{(0)})^{\g_{(0)}}$. 

If $f_1,f_2\in \gS(\wg_-^\vartheta)$, then there is 
$N'=n'm$ such that  $f_1,f_2\in \mathbb W_{N'}$. Set $n=2n'$. 
Then $\{f_1,f_2\}=0$ if and only if the images $\bar f_1,\bar f_2$ of $f_1, f_2$ in 
$\gS(\g)\simeq \gS(\wg^\theta)/(t^{-N}-1)$  
Poisson-commute.  The $\Z_N$-grading of $\g$ inherited from $\wh_-$ coincides with the 
$\Z_N$-grading defined by $\tilde\vth$. If $f_1,f_2\in \gZ(\wh_-,\vth)$, then 
 $\bar f_1,\bar f_2\in \gS(\g_{(0)})^{\g_{(0)}}$ by the construction.  
The assumption  $\ind\h _{(0)}=\rk\h $ implies that  $\ind\g_{(0)}=\rk\g$, see Lemma~\ref{twist}.
Thus $\gS(\g_{(0)})^{\g_{(0)}}$ is contained in the Poisson-commutative 
subalgebra $\gZ(\g,\tilde\vth)\subset \gS(\g)$ and hence $\{\bar f_1,\bar f_2\}=0$. 

{\sf (ii)} The algebra $\gZ(\wh_-,\vth)=\varinjlim \gS(\mathbb W_{nm})^{\h [t]^{\vartheta}}$ has  
a direct limit structure. If $\gZ(\wh_-,\vth)$ is a polynomial ring, then 
$\gS(\mathbb W_{N})^{\h [t]^{\vartheta}}\simeq\gS(\g_{(0)})^{\g_{(0)}}$ has to be a polynomial ring for each $n$. 

In order to prove the opposite implication, suppose that  $\gS(\mathbb W_{N})^{\h [t]^{\vartheta}}$
is a polynomial ring for each $n$. By a standard argument on graded algebras,  see e.g. the proof of Lemma~\ref{kosiks-lemma}, 
any algebraically independent set of generators of $\gS(\mathbb W_N)^{\h [t]^{\vartheta}}$ extends 
to an algebraically independent set of generators of $\gS(\mathbb W_{N{+}m})^{\h [t]^{\vartheta}}$.
In  the direct limit, one obtains a set of algebraically independent  generators of  $\gZ(\wh_-,\vth)$.
\end{proof}

If $\ind\h _{(0)}=\rk\h $, there is a {\sf g.g.s.} for $\vartheta$, and $\h _{(0)}$ has the {\sl codim}--$2$ 
property, then $\gS(\g_{(0)})^{\g_{(0)}}$ is a polynomial ring, see Lemma~\ref{twist} and 
Theorem~\ref{thm:kot14}. Set $b_{i,\bullet}=b_i(m-1)-b_i^\bullet$. 

\begin{prop}        \label{twist-p}
If \/ $\ind\h _{(0)}=\rk\h $, there is a {\sf g.g.s.} $F_1,\ldots,F_r\in\gS(\h)^{\h}$ for $\vartheta$, and $\h _{(0)}$ has the  {\sl codim}--$2$ property, then 
$\gZ(\wh_-,\vartheta)$ is freely generated by the 
$(t^{-1})$-polarisations $(F_i)_ {[b_{i,\bullet}+km]}$ with $k\ge 0$ related to the subspaces $((\h  t^{-j})^{\vartheta})^*$. 
\end{prop}
\begin{proof}
First we enlarge on $t^{-1}$-polarisations. For $k\in \mathbb Z$, let $\bar k\in\{0,\ldots,m{-}1\}$ be the 
residue of $k$ modulo $m$.  Then $\wh_-^{\vartheta}=\bigoplus_{k\le -1} \h _{\bar k} t^{k}$. 
Let $\phi_s\!: (\wh_-^\vartheta)^* \to (\wh_-^\vartheta)^*$ with $s\in\mK^\star$ be the linear map  multiplying elements of 
$(\h _{\bar k}t^{k})^*$ with $s^{-k-1}$. Next we canonically identify $(\h _{\bar k} t^{k})^*$ with 
$\h _{\bar k}^*$. Then an element of $(\wg_-^\vartheta)^*$ is a finite 
sequence $\xi=(\xi_{-1},\xi_{-2},\ldots,\xi_{-L})$ with $\xi_{k}\in \h _{\bar k}^*$. 
For such a  $\xi$, set $|\xi|= \sum_{i=1}^L \xi_{-i} \in \h ^*$. 
If $F\in\gS(\h )$, then $F(|\phi_s(\xi)|)=\sum_{k\ge 0}  s^kF_{[k]}(\xi)$, where the sum is actually finite. If $n=1$, then  $\mathbb W_m\simeq\gt h_0$ and $(F_i)_{[b_{i,\bullet}]}$ identifies with $F_i^\bullet\in\gS(\h_{(0)})^{\h_{(0)}}$.

Recall that $\gZ(\wh_-,\vth)=\varinjlim \gS(\mathbb W_{nm})^{\h [t]^{\vartheta}}$
and that $ \gS(\mathbb W_{N})^{\h [t]^{\vartheta}}\simeq \gS(\g_{(0)})^{\g_{(0)}}$ for 
$N=nm$. Now we identify the vector spaces $\mathbb W_N$ and $\g$. For  $0\le k<n$, 
let $F_{i}^{[k]}\in\gS(\g)$ be defined by \eqref{ik-tilde}. Then 
$\gS(\g_{(0)})^{\g_{(0)}}$ is freely generated by the highest components $(F_{i}^{[k]})^\bullet$ 
with $1\le i\le r$, $0\le k<n$. 
The discussion in the proof of Lemma~\ref{twist}{\sf (iii)} shows that if  
$B_i^\bullet -\ell_i+mk\in N\Z$, then $(F_i^{[k]})^\bullet=\frac{1}{n^{b_i}} (F_i)_{[b_{i,\bullet}]}\in\gS(\mathbb W_m)$. 
Analogously to the case of $\vth=\id_{\gt h}$, cf. Proposition~\ref{pol}, 
one can see that  
\[
    \{n^{b_i}(F_{i}^{[k]})^\bullet \mid 0\le k < n\}=\{(F_i)_{[b_{i,\bullet}+km]} \mid 0\le k<n\}.
\]
This completes the proof.  
\end{proof} 

\begin{rmk}
Theorem~\ref{twist-t} emphasises the importance of the equality $\ind\h_{(0)}=\rk\h$. 
Let us recall that it holds for the involutions by~\cite{p07}. Furthermore, if $\h$ is a classical Lie algebra 
and $\vth$ is an involution, then there is a {\sf g.g.s.} for $\vartheta$ and $\h _{(0)}$ has the 
{\sl codim}--$2$ property \cite{p07,contr}, i.e., Proposition~\ref{twist-p} applies. This means
that the $\vth$-twisted Poisson  analogue of the Feigin--Frenkel centre exists for many automorphisms 
and indicates that probably a $\vth$-twisted version of $\gt z(\wg)$ can be constructed 
at least for some $\vth$. Often both $\vth$-twisted objects are expected to be  
polynomial rings in infinite number of variables. Finally, we mention that results of~\cite{kot-T} 
can be used for the description of $\gS(\g_{(0)})^{\g_{(0)}}$. 
\end{rmk}

\end{document}